\documentclass[9pt,shortpaper,twoside,web]{ieeecolor}
\usepackage{generic}
\usepackage{graphicx}
\usepackage{cite}
\usepackage{amsmath,amssymb,amsfonts}
\usepackage{graphicx}
\usepackage{textcomp}
\usepackage{subfloat}
\usepackage{amsmath,amsfonts,amssymb}
\usepackage{amsthm}
\usepackage{algorithm}
\usepackage{algpseudocode}
\usepackage{subfig}
\usepackage{caption}
\usepackage{subcaption}
\usepackage{amsmath,amsfonts,amssymb}
\usepackage{amsthm}
\usepackage{hyperref}
\hypersetup{hidelinks=true}

\theoremstyle{definition}
\newtheorem{thm}{Theorem}
\usepackage{amsfonts} 
\usepackage{stackengine}
\newtheorem{definition}{Definition}
\newtheorem{assumption}{Assumption}
\newtheorem{lemma}{Lemma}
\usepackage{textcomp}
\usepackage{amsthm}
\usepackage{amsmath,mathtools}
\theoremstyle{definition}

\theoremstyle{remark}
\newtheorem*{remark}{Remark}
\def\BibTeX{{\rm B\kern-.05em{\sc i\kern-.025em b}\kern-.08em
    T\kern-.1667em\lower.7ex\hbox{E}\kern-.125emX}}
\begin{document}
\title{
Resilient Two-Time-Scale Local Stochastic Gradient Descent for Byzantine Federated Learning}
\author{Amit Dutta and Thinh T. Doan\thanks{Amit Dutta is with the Electrical and Computer Engineering Department at Virgnia Tech, email: amitdutta@vt.edu. Thinh T. Doan is with the Aerospace Engineering and Engineering Mechanics Department at University of Texas, Austin, email: thinhdoan@utexas.edu. This work was partially supported by NSF-CAREER Grant No. 2339509 and AFOSR YIP Grant No. 420525}
}

\maketitle

\newcommand{\tdoan}[1]{{\color{red}\bf [thinh: #1]}}

\begin{abstract}
We study local stochastic gradient descent methods for solving federated optimization over a network of agents communicating indirectly through a centralized coordinator. We are interested in the Byzantine setting where there is a subset of $f$ malicious agents that could observe the entire network and send arbitrary values to the coordinator to disrupt the performance of other non-faulty agents. The objective of the non-faulty agents is to collaboratively compute the optimizer of their respective local functions under the presence of Byzantine agents. In this setting, prior works show that the local stochastic gradient descent method can only return an approximate of the desired solutions due to the impacts of Byzantine agents. Whether this method can find an exact solution remains an open question. In this paper, we will address this open question by proposing a new variant of the local stochastic gradient descent method. Under similar conditions that are considered in the existing works, we will show that the proposed method converges exactly to the desired solutions. We will provide theoretical results to characterize the convergence properties of our method, in particular, the proposed method convergences at an optimal rate $\mathcal{O}(1/k)$ in both strongly convex and non-convex settings, where $k$ is the number of iterations. Finally, we will present a number of simulations to illustrate our theoretical results. 

\end{abstract}


\section{Introduction}
We consider a distributed optimization framework where there are $N$ agents communicating with a single coordinator. This framework is also popularly known as {\em federated optimization}~\cite{konevcny2015federated}. Associated with each agent $i$ is a function $q^{i}:\mathbb{R}^{d}\rightarrow \mathbb{R}$. The goal of the agents is to find a point $x^{\star}$ that optimizes their aggregate local functions.


Besides traditional machine learning applications\cite{Li_survery2020}, federated optimization now also finds application in networked systems e.g., {internet of vehicles}~\cite{manias2021making}, industrial control systems~\cite{huong2021detecting}, and wireless \cite{niknam2020federated, chen2021distributed, khan2021federated}.
One of the key advantages of the federated optimization framework is its ability to implement optimization algorithm updates locally at the agents without necessitating the transmission of raw data to a centralized coordinator.
This localized data processing not only reduces the communication overhead between agents and the central server but also introduces an element of privacy preservation.

One of the main challenges in federated learning is the vulnerability of the system to malicious attacks where some agents in the network may fail or whose updates can be manipulated by an external entity. Such malicious agents will have detrimental impacts to the performance of other agents, and if not addressed, it can lead to catastrophic failures of entire network. For example, malicious attacks have been identified as the most critical problem in wireless spectrum sensing \cite{shi2021challenges}.    

In this paper, we are interested in studying the so-called distributed local stochastic gradient descent (SGD) for solving federated optimization. Our focus is to characterize the performance of this method when there are a (small) number of Byzantine malicious agents in the network . In this setting,  Byzantine agents can observe the entire network and send any information to the centralized coordinator to corrupt the output of local SGD. Due to the impact of Byzantine agents, our goal now is to solve the optimization problem that only involves the honest agents. In particular, we consider the setting where there are up to $f$ faulty Byzantine agents with unknown identities. We then seek to solve an exact fault-tolerance problem defined as follows.

\textbf{Exact fault-tolerance problem}: 
Let $\mathcal{H}$ be the set of honest agents with $|\mathcal{H}| \geq N-f$, then an algorithm is said to have exact fault-tolerance if it allows all the non-faulty agents to compute
\begin{align}\label{eq:exact_tolerence}
    x_{\mathcal{H}}^{\star} \in \arg \min_{x\in \mathbb{R}^{d}}\sum_{i \in \mathcal{H}}q^{i}(x).
\end{align}
Note that if the number of Byzantine agent is large, i.e., $f> \mathcal{H}$, it is impossible to solve problem \eqref{eq:exact_tolerence}. We, therefore, consider the following $2f$-\textit{redundancy} condition, which is necessary and sufficient for solving problem \eqref{eq:exact_tolerence} exactly \cite{gupta2020fault}. 
\begin{definition}[$2f$-\textbf{redundancy}]
The set $\mathcal{H}$, with $|\mathcal{H}| \geq N-f$, is said to have $2f$-redundancy if for any subset $\mathcal{S} \subset \mathcal{H}$ with $|\mathcal{S}| \geq N -2f$
\begin{align}
    \arg \min_{x\in \mathbb{R}^{d}}\sum_{i \in \mathcal{S}}q^{i}(x) = \arg \min_{x\in \mathbb{R}^{d}}\sum_{i \in \mathcal{H}}q^{i}(x).
\end{align}
\end{definition}
We note that the $2f$-redundancy condition arises naturally in many problems, including hypothesis testing \cite{chong2015observability, gupta2019byzantine, mishra2016secure, su2019finite}, and distributed learning \cite{alistarh2018byzantine, blanchard2017machine, charikar2017learning, guerraoui2018hidden}. For example, in distributed learning this condition is satisfied when all agents have identical objective functions, the so-called homogeneous setting. 


In this paper, we will investigate the convergence of local SGD in solving problem \eqref{eq:exact_tolerence} under this $2f$-redundancy condition. In \cite{gupta2023byzantine}, the authors show that the local deterministic gradient descent can find $x_{\mathcal{H}}^{\star}$ when using a comparative elimination (CE) filter in its update to address the impacts of Byzantine agents. However, their approach cannot be extended to the case of local SGD, i.e., each agent only has access to stochastic samples of $\nabla q^{i}(\cdot)$. In this stochastic setting, the work in \cite{gupta2019byzantine} can only return a point within a ball around $x_{\mathcal{H}}^{\star}$ whose the size depends on the ratio $f/\mathcal{H}$. The level idea behind this issue is that the CE filter cannot simultaneously address the impacts of Byzantine agents and local stochastic errors due to gradient sampling. Our focus is therefore to address this open question. Specifically, we will propose a new variant of local SGD that will allow each agent to find exactly $x_{\mathcal{H}}^{\star}$ in both strongly convex and non-convex settings. Our main contribution is summarized as follows.

\noindent\textbf{Main contribution.} We propose a new two-time-scale variant of local SGD for solving problem \eqref{eq:exact_tolerence} in the Byzantine setting under the $2f$-redundancy condition. We will show that the proposed algorithm can return an exact solution $x_{\mathcal{H}}^{\star}$ of problem \eqref{eq:exact_tolerence}. In addition, we will study theoretical results to characterize the convergence rate of our algorithms when the underlying objective function satisfies either strong convexity or non-convex Polyak-\L ojasiewicz (P\L) condition. In both cases, our algorithm converges to the optimal solutions at an optimal rate $\mathcal{O}(1/k)$, where $k$ is the number of iteration. Finally, we will provide a few numerical simulations to illustrate the correctness of our theoretical results.  

\subsection{Related work}
According to existing literature, there are various Byzantine fault-tolerant aggregation schemes for distributed optimization and learning. These include \textit{multi-KRUM} \cite{blanchard2017machine}, \textit{coordinate-wise trimmed mean} (CWTM) \cite{su2019finite}, \textit{geometric median-of-means}(GMoM) \cite{chen2017distributed}, \textit{minimum-diameter averaging} (MDA) \cite{guerraoui2018hidden}, and \textit{Byzantine-robust stochastic aggregation} (RSA) \cite{li2019rsa} filters. However, it is important to note that these schemes do not guarantee \textit{exact fault-tolerance} even in a deterministic setting with $2f$-redundancy, unless additional assumptions are made regarding the objectives of the honest agents. The work by \cite{gupta2023byzantine} shows that it is possible to achieve exact fault tolerance in deterministic setting and approximate fault tolerance in stochastic setting in $2f$-redundancy scenario. \cite{farhadkhani2022byzantine} proposed \textit{RESilient Averaging of Momentum} (RESAM) which presents an unified byzantine fault-tolerant framework with accelerated gradient descent based on the previously mentioned methods. They also established finite time convergence with some additional assumptions. Although their results hold for non-convex objectives, they clearly proved that such generalization and acceleration cannot be applied for the CE aggregation scheme. Recently, \cite{liu2023impact} explored the impact of Byzantine agents and stragglers, where stragglers are agents that experience significant delays in their updates, on solving distributed optimization problems under the redundancy of cost functions. Furthermore, the authors examined Byzantine fault-tolerant min-max distributed optimization problems under similar redundancy conditions in \cite{liu2022byzantine}. Both works demonstrate that the authors' approach can only achieve approximate solutions. We also want to note some relevant work in \cite{cao2019distributed, cao2020distributed}, where the authors study approximate fault tolerance problem with more relaxed conditions on the Byzantine agents. Our work in this paper builds on the works of \cite{gupta2023byzantine} where based on the $2f$-redundancy condition we propose a two-time-scale variant of the local SGD in both strongly convex non-convex setting. To address the affect of approximate convergence to optimal solution in byzantine free case minibatch SGD has been studied in \cite{woodworth2020local}, \cite{woodworth2020minibatch} in order to reduce the dependency on the variance of the stochastic gradients. Here the authors have further proposed an accelerated version of the mini-batch SGD to further reduce the impact of gradient noise in both IID and non-IID sampling cases. However no such improvements have been addressed in literature when a given network is under attack from byzantine agents.

Another relevant literature to this paper is the recent works in studying the complexity of two-time-scale stochastic approximation, see for example \cite{ 8919880,GuptaSY2019_twoscale,Doan_two_time_SA2019,Kaledin_two_time_SA2020,doan2022nonlinear, zeng2021two,zeng2024fast, doan2024fast,zeng2022regularized}. It has been observed that two-time-scale approach can be used to either study or design better distributed algorithms in different settings, e.g., delays \cite{doan2017convergence}, quantization \cite{DoanMR2018b}, and cluster networks \cite{dutta2021convergence,9992395}. In this paper, we will leverage this idea to design a new resilient local SGD for the Byzantine setting.

\section{Resilient Two-Time-Scale Local SGD}\label{sec:Algorithm_description}
\begin{algorithm}[!t]
\caption{Resilient Two-Time-Scale Local SGD}\label{alg:cap}
\begin{algorithmic}[1]
    \State \textbf{Initialize}: The server initializes the model with $\Bar{x}_{0}\in \mathbb{R}^{d}$. Each agent initializes with step-sizes $y^{i}_{0,0}, \alpha_{k}, \beta_{k}$ and chooses $\mathcal{T}$.
\For{$k=1,..$}
    \State All clients $i=1,2,..,N$ in parallel do
    \State Receive $\Bar{x}_{k}$ from the server and set $x^{i}_{k,0} = \Bar{x}_{k}$
    \State Agents perform initialization $y^{i}_{k,0} = y^{i}_{k-1,\mathcal{T}}$
     \For{$t=0,...,\mathcal{T}-1$}
     \begin{align}
        x^{i}_{k,t+1} &= x^{i}_{k,t} - \beta_{k} y^{i}_{k,t},\label{alg:xi}\\ 
        y^{i}_{k,t+1} &= (1-\alpha_{k})y^{i}_{k,t} + \alpha_{k}\nabla q^{i}(x^{i}_{k,t};\Delta ^{i}_{k,t}).\label{alg:yi}
     \end{align}
    \EndFor
    \State Agent $i$ broadcasts $x^{i}_{k,\mathcal{T}}$ to the server. Server arranges the distance between $x^{i}_{k,\mathcal{T}}$ and $\Bar{x}_{k}$ in ascending order as
     \begin{align}
         \|\Bar{x}_{k} - x^{i_1}_{k,\mathcal{T}}\| \leq \|\Bar{x}_{k} - x^{i_2}_{k,\mathcal{T}}\| \leq \cdots \leq \|\Bar{x}_{k} - x^{i_N}_{k,\mathcal{T}}\|.\label{alg:sort_distances}
     \end{align}
    \State Server eliminate the $f$ -largest distances to obtain 
    $$\mathcal{F}_{k} =\{i_{1}, \cdots,i_{N-f}\}.$$ 
    \State Using the estimates of clients in $\mathcal{F}_{k}$, the server implements
     \begin{align}
     \Bar{x}_{k+1} = \frac{1}{|\mathcal{F}_{k}|}\sum_{i\in \mathcal{F}_{k}}^{}x^{i}_{k,\mathcal{T}}.\label{alg:xbar}
     \end{align}
\EndFor
\end{algorithmic}

\end{algorithm}
In this section, we present the proposed algorithm, namely, resilient two-time scale local SGD, for solving problem \eqref{eq:exact_tolerence} when there are up to $f$ Byzantine agents. Our algorithm is formally stated in Algorithm \ref{alg:cap}, where  each honest agent $i\in\mathcal{H}$ maintains two local variables $x^{i},y^{i}$ to estimate the optimal value $x_{\mathcal{H}}^{\star}$ and its local gradient $\nabla q^{i}(\cdot)$, respectively. On the other hand, the server maintains a global variable $\bar{x}$ to estimate the average of the iterates sent by the agents. At any global iteration $k$, each agent $i\in\mathcal{H}$ implements $\mathcal{T}$ two-time-scale SGD steps to update its variables (a.k.a Eqs. \eqref{alg:xi} and \eqref{alg:yi}), where $\nabla q^{i}(\cdot;\Delta^{i})$ is the sample of $\nabla q^{i}(\cdot)$. These two updates are implemented by using two different step sizes $\alpha_{k} \geq \beta_{k}$, i.e., the update of $y^{i}$ is implemented at a ``faster" time scale than $x^{i}$, explaining for the name of two-time-scale SGD. In particular, each agent first estimates its local gradient from the samples, which is then used to update its local variable toward the optimal solution $x_{\mathcal{H}}^{\star}$. After $\mathcal{T}$ steps, each honest agent will send its last iterate $x_{k,\mathcal{T}}^{i}$ to the server. We note that Byzantine agents can send any arbitrarily to the server. To address this issue, the server will implement a comparative elimination (CE) filter, studied in \cite{gupta2020fault}. In particular, the server first sorts the distance between the agent estimates and its average in an ascending order (Step 8). The server then eliminates the $f$ largest distances (Step 9), i.e., since it does not know the identity of Byzantine agents the server can only eliminate any ``suspiciously large values". The server then computes a new average based on the estimates of remaining $N-f$ agents in Eq. \eqref{alg:xbar}.

In Eq. \eqref{alg:yi}, when $\alpha_{k} = 1$ the proposed algorithm is reduced to the local SGD with CE filter studied in \cite{gupta2023byzantine}. However, as we will show in this paper, by properly choosing $\alpha_{k}$ and $\beta_{k}$ our algorithm will guarantee an exact convergence to $x_{\mathcal{H}}^{\star}$ even under the Byzantine setting, which cannot be achieved by the one in \cite{gupta2023byzantine}. In particular, we choose the step sizes to satisfy $\beta_{k} \leq \alpha_{k} \leq 1$ as follows
\begin{align}\label{eq:step_sizes}
\begin{aligned}
    &\alpha_{k} = \frac{C_\alpha}{1+h+k}, \quad \beta_{k} = \frac{C_\beta}{1+h+k},\\ 
    &\text{where }\; C_\beta \leq C_\alpha\; \text{ s.t. } \alpha_k\leq 1 \text{ and } \;L\mathcal{T} \beta_{k}\leq 1,\;\;\forall k\geq0.
    \end{aligned}
\end{align}
We will demonstrate in Theorems \ref{Theorem:SC_Tg1} and \ref{Theorem:PL_Tg1} that, with an appropriate selection of the parameter $h > 1$, this choice of step sizes is essential for the algorithm to achieve an optimal convergence rate of $\mathcal{O}(1/k)$. 

\section{Technical Assumptions and Preliminaries}
We present here the main technical assumptions and some preliminaries, which will help to facilitate the development of our main results later. First, we consider the following two main assumptions, where $\mathcal{P}_{k,t}$ is the filtration that includes all the random variables generated by Algorithm \ref{alg:cap} up to time $k+t$. 

\begin{assumption}\label{eq:Assumption_noise}
    The random variables $\Delta^{i}_{k}$, $\forall i$ and $k\geq 0$, are i.i.d. and there exists a positive constant $\sigma$ such that we have $\forall x \in \mathbb{R}^{d}$
    \begin{align*}
    \begin{aligned}
        &\mathbb{E}[\nabla q^{i}(x,\Delta^{i}_{k,t})\mid \mathcal{P}_{k,t}] = \nabla q^{i}(x),\\
        &\mathbb{E}[\|\nabla q^{i}(x,\Delta^{i}_{k,t})-\nabla q^{i}(x)\|^{2}\mid \mathcal{P}_{k,t}] \leq \sigma^{2}.
        \end{aligned}
    \end{align*}
\end{assumption}

\begin{assumption}\label{eq:Assumption_Lipschitz} For each $i\in\mathcal{H}$, $q^{i}$ has Lispchitz continuous gradients, i.e., there exists a constant $L>0$ such that 
\begin{align*}
    \|\nabla q^{i}(y) - \nabla q^{i}(x)\| \leq L \|y-x\|, \ \forall x,y \in \mathbb{R}^{d}.
\end{align*}
\end{assumption}
In the sequel, we will assume that these two assumptions always hold. For notational convenience, we define
\begin{align*}
     q_{\mathcal{H}}(x) = \frac{1}{|\mathcal{H}|}\sum_{i \in \mathcal{H}}q^{i}(x).
\end{align*}
Let $e_{k,t}^{i}$ be the local gradient estimate error at client $i$ defined as
\begin{align}
    e_{k,t}^{i} = y_{k,t}^{i} - \nabla q^{i} (\bar{x}_{k}),\label{notation:e_i} 
\end{align}
and $W_{k}$ be the average of local gradient estimate error defined as 
\begin{align}\label{eq:V_2_sc}
    W_{k} = \frac{1}{|\mathcal{H}|}\sum_{i \in \mathcal{H}}\|e^{i}_{k,0}\|^{2}.
\end{align}
We denote by $\mathcal{B}_{k}$ the set of potential Byzantine clients and $\mathcal{H}_{k}$ the set of non-faulty clients in $\mathcal{F}_{k}$ in step $9$ in Algorithm \ref{alg:cap}. Finally, we denote by $\mathcal{X}^{\star}_{\mathcal{H}}$ and $\mathcal{X}^{\star}_{i}$ the sets of minimizers of $q_{\mathcal{H}}$ and $q^{i}$, respectively. Then, under Assumption \ref{eq:Assumption_Lipschitz} and the $2f$-redundancy property, one can show that \cite[Lemma 1]{gupta2020resilience}
\begin{align}\label{eq:redundancy_condition}
    \bigcap \limits_{i\in \mathcal{H}} \mathcal{X}^{\star}_{i} = \mathcal{X}^{\star}_{\mathcal{H}}. 
\end{align}
Using the notation above, we rewrite \eqref{alg:xi} as follows: $\forall i\in\mathcal{H}$
\begin{align}
\label{eq:x_update_T}
    x^{i}_{k,\mathcal{T}} 
    & = \Bar{x}_{k} -\beta_{k}\sum_{t=0}^{\mathcal{T}-1}e^{i}_{k,t} -\mathcal{T}\beta_{k}\nabla q^{i}(\bar{x}_{k}),
\end{align}
which since $|\mathcal{F}_{k}| = |\mathcal{H}|$ and $|\mathcal{B}_{k}| = |\mathcal{H}\backslash\mathcal{H}_{k}|$ yields
\begin{align}\label{eq:x_bar_update_0}
        \Bar{x}_{k+1}&=\frac{1}{|\mathcal{F}_{k}|}\sum_{i \in \mathcal{F}_{k}}x^{i}_{k,\mathcal{T}}\notag \\
        &=\frac{1}{|\mathcal{H}|}\Big[\sum_{i\in \mathcal{H}}x^{i}_{k,\mathcal{T}} + \sum_{i\in \mathcal{B}_{k}}x^{i}_{k,\mathcal{T}} - \sum_{i\in \mathcal{H} \backslash \mathcal{H}_{k}}x^{i}_{k,\mathcal{T}}\Big],\notag\\
        &= \bar{x}_{k} - \mathcal{T}\beta_{k}\nabla q_{\mathcal{H}}(\bar{x}_{k}) - \frac{\beta_{k}}{|\mathcal{H}|}\sum_{i \in \mathcal{H}}\sum_{t=0}^{\mathcal{T}-1}e^{i}_{k,t} \notag\\
        &\quad + \frac{1}{|\mathcal{H}|}\Big[\sum_{i\in \mathcal{B}_{k}}(x^{i}_{k,\mathcal{T}} -\Bar{x}_{k}) - \sum_{i\in \mathcal{H} \backslash \mathcal{H}_{k}}(x^{i}_{k,\mathcal{T}}-\Bar{x}_{k})\Big].
    \end{align}
We next consider the following result on $W_k$, where for an ease of exposition its proof will be presented in the Appendix.  
\begin{lemma}\label{lem:Lemma_V2_SC}
For all $k\geq 0$ we have 
    \begin{align}
           \mathbb{E}[W_{k+1}]
    &\leq \Big(1-\frac{\alpha_{k}}{2} + 126(L+1)^{2}\mathcal{T}^{2}\beta_{k}\Big)\mathbb{E}[W_{k}]\notag\\
    &\quad +\Big(14L^{3}\mathcal{T}^{2}\alpha_{k}\beta_{k} + \frac{300L^{4}\mathcal{T}^{2}\beta^{2}_{k}}{\alpha_{k}} \Big)\mathbb{E}[\|\bar{x}_{k}-x^{\star}_{\mathcal{H}}\|^{2}]\notag\\
     &\quad+ 2\sigma^{2}\mathcal{T}\alpha^{2}_{k} + 96L^{2}\mathcal{T}^{2}\sigma^{2}\beta^{2}_{k}\notag\\
    &\quad+ 6L^{2}\mathcal{T}^{2}\sigma^{2}\alpha^{2}_{k}\beta^{2}_{k} + \frac{96L^{2}\mathcal{T}^{2}\sigma^{2}f\alpha_{k}\beta^{2}_{k}}{|\mathcal{H}|}\cdot\label{lem:Lemma_V2_SC:ineq}
    \end{align}
\end{lemma}

\section{Main Results}
In this section, we present the main results of this paper, where we will study the convergence properties of Algorithm \ref{alg:cap} in two settings, namely, strong convexity and PL conditions. 
\subsection{Strongly convex condition}
We consider the following assumption on $q_{\mathcal{H}}(x)$.
\begin{assumption}\label{Assum:Assumption_strong_convexity}
The objective function $q_{\mathcal{H}}(.)$ is strongly convex, i.e., there exists a constant $\mu \in (0,L]$ s.t.
\begin{align}\label{eq:sc_condition}
    (y-x)^{T}(\nabla q_{\mathcal{H}}(y) - \nabla q_{\mathcal{H}}(x)) \geq \mu \|y-x\|^{2}, \ \forall x,y \in \mathbb{R}^{d}.
\end{align}
\end{assumption}
\begin{remark}
Assumption \ref{Assum:Assumption_strong_convexity} guarantees a unique solution \( x_\mathcal{H}^\star \) for problem \eqref{eq:exact_tolerence}. However, it does not require each local function \( q^i \) to be strongly convex, meaning each \( q^i \) may have multiple minimizers. Note that under the $2f$-redundancy condition, we have \( x_\mathcal{H}^\star \) lies in the intersection of the minimizer sets $\mathcal{X}_{i}^{\star}$ of  $q^i$.
\end{remark}
Our main result in this section is based on the following lemma, whose proof is presented in the Appendix.
\begin{lemma}\label{lem:Lemma_V1_SC}
For all $k \geq 0$ we have
    \begin{align}
    &\mathbb{E}[\|\bar{x}_{k+1}-x^{\star}_{\mathcal{H}}\|^{2}]\notag\\
    & \leq \Big(1-\frac{35\mu\mathcal{T}\beta_{k}}{18} + \frac{17L\mathcal{T}\beta_{k}|\mathcal{B}_{k}|}{3|\mathcal{H}|}+ 103L^{2}\mathcal{T}^{2}\beta^{2}_{k}\Big)\mathbb{E}[\|\bar{x}_{k}-x^{\star}_{\mathcal{H}}\|^{2}]\notag\\
    &\quad + \frac{72L^{2}\mathcal{T}\alpha_{k}\beta_{k}}{\mu}\mathbb{E}[\|\bar{x}_{k}-x^{\star}_{\mathcal{H}}\|^{2}] +  \frac{78\mathcal{T}\beta_{k}}{\mu}\mathbb{E}[W_{k}]\notag\\
    &\quad + 4\mathcal{T}^{2}\sigma^{2}\alpha_{k}\beta^{2}_{k} + 
    \frac{48\mathcal{T}\sigma^{2}\alpha^{2}_{k}\beta_{k}}{\mu} + \frac{32\mathcal{T}^{2}\sigma^{2}f\alpha^{2}_{k}\beta^{2}_{k}}{|\mathcal{H}|}\cdot\label{lem:Lemma_V1_SC:ineq}
\end{align}
\end{lemma}
To study the convergence of Algorithm \ref{alg:cap} under Assumption \ref{Assum:Assumption_strong_convexity}, we consider the following Lyapunov function $V_{k}$ 
\begin{align}\label{eq:Lyapunov_SC}
    V_{k} = \|\bar{x}_{k} - x^{\star}_{\mathcal{H}}\|^{2} + W_{k}.
\end{align}

\begin{thm}\label{Theorem:SC_Tg1}
    Let $\{x^{i}_{k}\}$ and $\{y^{i}_{k}\}$ be generated by Algorithm \ref{alg:cap} for $\mathcal{T}>1$. 
    Suppose the step sizes $\alpha_{k}$ and $\beta_{k}$ in \eqref{eq:step_sizes} satisfying 
\begin{align}\label{eq:alpha_beta_choice_SC}
\begin{aligned}
         \alpha_{k} \leq \frac{\mu}{(8)^{4}(L+1)^{4}\mathcal{T}}, \quad \beta_{k} \leq \frac{\mu}{(12)^{4}L^{2}\mathcal{T}}, \\
    \frac{\beta_{k}}{\alpha_{k}} \leq \frac{\mu}{(14)^{4}(L+1)^{4}\mathcal{T}}\cdot
    \end{aligned}
    \end{align}
    Then we have
\begin{align}\label{eq:Gamma_Theorem}
    &\mathbb{E}[V_{k+1}]\notag\\
    &\leq \big(1-\frac{23\mu\mathcal{T}\beta_{k}}{12} + \frac{17L\mathcal{T}\beta_{k}|\mathcal{B}_{k}|}{3|\mathcal{H}|}\big)\mathbb{E}[V_k]\notag\\
    &\quad + \frac{150(L+1)^{3}\mathcal{T}^{2}\sigma^{2}\alpha^{2}_{k}}{\mu} + \frac{128(L+1)^{2}\mathcal{T}^{2}\sigma^{2}f\alpha^{2}_{k}}{|\mathcal{H}|}\cdot
\end{align}
In addition, if the following condition holds 
    \begin{align}\label{eq:filter_condition}
        \frac{|\mathcal{B}_{k}|}{|\mathcal{H}|} = \frac{f}{N-f} \leq \frac{\mu}{3L},
    \end{align}
  and $C_{\alpha}, C_{\beta}$ and $h$ are chosen as
\begin{align}\label{eq:step_size_parameter_SC}
\begin{aligned}
         & C_{\alpha} \geq \frac{(84)^{4}(L+1)^{4}}{6\mu^{2}}, \quad C_{\beta} = \frac{72}{\mu\mathcal{T}}\\
         & h \geq \max\Big\{\frac{(8)^{4}(L+1)^{4}\mathcal{T}C_{\alpha}}{\mu};\ \frac{(72)^{4}L^{2}}{18\mu^{2}}\Big\},
         \end{aligned}
    \end{align}
then we obtain the following.
\begin{align}\label{eq:convergence_rate_SC}
    \mathbb{E}[V_{k+1}] &\leq \frac{h^{2}\mathbb{E}[V_{0}]}{(1+h+k)^{2}} + \frac{150(L+1)^{3}\mathcal{T}^{2}\sigma^{2}C^{2}_{\alpha}}{\mu(1+h+k)} \notag\\
    &\quad+ \frac{128(L+1)^{2}\mathcal{T}^{2}\sigma^{2}fC^{2}_{\alpha}}{(1+h+k)|\mathcal{H}|}\cdot
\end{align}    
\end{thm}
\begin{remark}
Our result in \eqref{eq:convergence_rate_SC} implies that the sequence $\{\bar{x}_{k}\}$ generated by Algorithm \ref{alg:cap} converges to $x^{\star}_{\mathcal{H}}$ in the mean-square sense at a rate $\mathcal{O}(1/k)$, which is the same rate as in the Byzantine-free setting. In addition, our convergence complexity bound depends on the ratio $f/\mathcal{H}$, which is similar to the result in \cite{gupta2023byzantine}.  
\end{remark}
\begin{proof}
    By adding \eqref{lem:Lemma_V2_SC:ineq} into \eqref{lem:Lemma_V1_SC:ineq} we have
    \begin{align} 
        &\mathbb{E}[V_{k+1}]\notag\\
    &\leq \Big(1-\frac{35\mu\mathcal{T}\beta_{k}}{18} + \frac{17L\mathcal{T}\beta_{k}|\mathcal{B}_{k}|}{3|\mathcal{H}|} + 103L^{2}\mathcal{T}^{2}\beta^{2}_{k}\Big)\mathbb{E}[\|\bar{x}_{k}-x^{\star}_{\mathcal{H}}\|^{2}]\notag\\
    &\quad + \Big(\frac{300L^{4}\mathcal{T}^{2}\beta^{2}_{k}}{\alpha_{k}} + \frac{72L^{2}\mathcal{T}\alpha_{k}\beta_{k}}{\mu}\Big)\mathbb{E}[\|\bar{x}_{k}-x^{\star}_{\mathcal{H}}\|^{2}]\notag\\
    &\quad+ 14L^{3}\mathcal{T}^{2}\alpha_{k}\beta_{k}\mathbb{E}[\|\bar{x}_{k}-x^{\star}_{\mathcal{H}}\|^{2}]\notag\\ 
    &\quad+\Big(1-\frac{\alpha_{k}}{2} + 126(L+1)^{2}\mathcal{T}^{2}\beta_{k} + \frac{78\mathcal{T}\beta_{k}}{\mu}\Big)\mathbb{E}[W_{k}]\notag\\
    &\quad+ 2\mathcal{T}\sigma^{2}\alpha^{2}_{k} + 4\mathcal{T}^{2}\sigma^{2}\alpha_{k}\beta^{2}_{k} + \frac{42\mathcal{T}^{2}\sigma^{2}\alpha_{k}\beta_{k}}{\mu}\notag\\
    &\quad + 6L^{2}\mathcal{T}^{2}\sigma^{2}\alpha^{2}_{k}\beta^{2}_{k} + 96L^{2}\mathcal{T}^{2}\sigma^{2}\beta^{2}_{k}\notag\\
    &\quad + \frac{32\mathcal{T}^{2}\sigma^{2}f\alpha^{2}_{k}\beta^{2}_{k}}{|\mathcal{H}|} + \frac{96L^{2}\mathcal{T}^{2}\sigma^{2}f\alpha_{k}\beta^{2}_{k}}{|\mathcal{H}|}\allowdisplaybreaks\notag\\
     &\leq \Big(1-\frac{35\mu\mathcal{T}\beta_{k}}{18} + \frac{17L\mathcal{T}\beta_{k}|\mathcal{B}_{k}|}{3|\mathcal{H}|} + 103L^{2}\mathcal{T}^{2}\beta^{2}_{k}\Big)\mathbb{E}[\|\bar{x}_{k}-x^{\star}_{\mathcal{H}}\|^{2}]\notag\\
    &\quad + \Big(\frac{300L^{4}\mathcal{T}^{2}\beta^{2}_{k}}{\alpha_{k}} + \frac{72L^{2}\mathcal{T}\alpha_{k}\beta_{k}}{\mu} \Big)\mathbb{E}[\|\bar{x}_{k}-x^{\star}_{\mathcal{H}}\|^{2}]\notag\\
    &\quad+ 14L^{3}\mathcal{T}^{2}\alpha_{k}\beta_{k}\mathbb{E}[\|\bar{x}_{k}-x^{\star}_{\mathcal{H}}\|^{2}]\notag\\
    &\quad+\Big(1-\frac{\alpha_{k}}{2} + 126(L+1)^{2}\mathcal{T}^{2}\beta_{k} + \frac{78\mathcal{T}\beta_{k}}{\mu}\Big)\mathbb{E}[W_{k}]\notag\\
    &\quad + \frac{150(L+1)^{3}\mathcal{T}^{2}\sigma^{2}\alpha^{2}_{k}}{\mu} + \frac{128(L+1)^{2}\mathcal{T}^{2}\sigma^{2}f\alpha^{2}_{k}}{|\mathcal{H}|}\notag\\
    &\leq \Big(1-\frac{69\mu\mathcal{T}\beta_{k}}{36} + \frac{17L\mathcal{T}\beta_{k}|\mathcal{B}_{k}|}{3|\mathcal{H}|}\Big)\mathbb{E}[V_k] \notag\\
    &\quad + \Big(-\frac{\mu\mathcal{T}\beta_{k}}{36} + 103L^{2}\mathcal{T}^{2}\beta^{2}_{k} + 14L^{3}\mathcal{T}^{2}\alpha_{k}\beta_{k}\Big)\mathbb{E}[\|\bar{x}_{k} - x^{\star}_{\mathcal{H}}\|^{2}]\notag\\
    &\quad +\Big(\frac{72L^{2}\mathcal{T}\alpha_{k}\beta_{k}}{\mu} + \frac{300L^{4}\mathcal{T}^{2}\beta^{2}_{k}}{\alpha_{k}}\Big)\mathbb{E}[\|\bar{x}_{k} - x^{\star}_{\mathcal{H}}\|^{2}]\notag\\
    &\quad +\Big(-\frac{\alpha_{k}}{2} + \frac{69\mu\mathcal{T}\beta_{k}}{36} + 126(L+1)^{2}\mathcal{T}\beta_{k} + \frac{78\mathcal{T}\beta_{k}}{\mu}\Big)\mathbb{E}[W_{k}]\notag\\
    &\quad + \frac{150(L+1)^{3}\mathcal{T}^{2}\sigma^{2}\alpha^{2}_{k}}{\mu} + \frac{128(L+1)^{2}\mathcal{T}^{2}\sigma^{2}f\alpha^{2}_{k}}{|\mathcal{H}|},\notag
    \end{align}
    where the second inequality is obtained using $\beta_{k}\leq \alpha_{k}\leq 1$ and $\mu \leq L$. 
    Using $\mu \leq L$ we express the above inequality as
    \begin{align} 
        &\mathbb{E}[V_{k+1}]\notag\\
        &\leq \Big(1-\frac{69\mu\mathcal{T}\beta_{k}}{36} + \frac{17L\mathcal{T}\beta_{k}|\mathcal{B}_{k}|}{3|\mathcal{H}|}\Big)\mathbb{E}[V_k] \notag\\
        &\quad+\Big(-\frac{\mu\mathcal{T}\beta_{k}}{36} +103L^{2}\mathcal{T}^{2}\beta^{2}_{k}\Big)\mathbb{E}[\|\bar{x}_{k}-x^{\star}_{k}\|^{2}]\notag\\
        &\quad+ \Big(\frac{86(L+1)^{4}\mathcal{T}^{2}\alpha_{k}\beta_{k}}{\mu} + \frac{300(L+1)^{4}\mathcal{T}\beta_{k}}{\mu}\Big)\mathbb{E}[\|\bar{x}_{k}-x^{\star}_{\mathcal{H}}\|^{2}]\notag\\
        &\quad+\Big(-\frac{\alpha_{k}}{2} +\frac{206(L+1)^{4}\mathcal{T}\beta_{k}}{\mu}\Big)\mathbb{E}[W_{k}]\notag\\
        &\quad + \frac{150(L+1)^{3}\mathcal{T}^{2}\sigma^{2}\alpha^{2}_{k}}{\mu} + \frac{128(L+1)^{2}\mathcal{T}^{2}\sigma^{2}f\alpha^{2}_{k}}{|\mathcal{H}|}\notag\\
        &\leq \Big(1-\frac{69\mu\mathcal{T}\beta_{k}}{36} + \frac{17L\mathcal{T}\beta_{k}|\mathcal{B}_{k}|}{3|\mathcal{H}|}\Big)\mathbb{E}[V_k] \notag\\
        &\quad + \frac{150(L+1)^{3}\mathcal{T}^{2}\sigma^{2}\alpha^{2}_{k}}{\mu} + \frac{128(L+1)^{2}\mathcal{T}^{2}\sigma^{2}f\alpha^{2}_{k}}{|\mathcal{H}|},
        \notag
    \end{align}
where the last inequality we use \eqref{eq:alpha_beta_choice_SC} to have 
\begin{align}
        &0 \leq \frac{\mu\mathcal{T}\beta_{k}}{36} -103L^{2}\mathcal{T}^{2}\beta^{2}_{k} - \frac{86(L+1)^{4}\mathcal{T}^{2}\alpha_{k}\beta_{k}}{\mu}  \notag\\
        &\qquad -\frac{300(L+1)^{4}\mathcal{T}^{2}\beta^{2}_{k}}{\alpha_{k}},\notag\\
        &0\leq  \frac{\alpha_{k}}{2} - \frac{206(L+1)^{4}\mathcal{T}\beta_{k}}{\mu}\cdot\notag
    \end{align}


Next, to show \eqref{eq:convergence_rate_SC} we observe that the conditions \eqref{eq:step_size_parameter_SC} satisfy those in \eqref{eq:alpha_beta_choice_SC}. Thus, we have
\begin{align}
    &\mathbb{E}[V_{k+1}]\notag\\
    &\leq \Big(1-\mathcal{T}\beta_{k}\Big(\frac{23\mu}{12} - \frac{17L|\mathcal{B}_{k}|}{3|\mathcal{H}|}\Big)\Big)\mathbb{E}[V_k]\notag\\
    &\quad + \frac{150(L+1)^{3}\mathcal{T}^{2}\sigma^{2}\alpha^{2}_{k}}{\mu} + \frac{128(L+1)^{2}\mathcal{T}^{2}\sigma^{2}f\alpha^{2}_{k}}{|\mathcal{H}|}.
    \notag\\
    &\leq \Big(1-\frac{\mu\mathcal{T}\beta_{k}}{36}\Big)\mathbb{E}[V_k]\notag\\
    &\quad + \frac{150(L+1)^{3}\mathcal{T}^{2}\sigma^{2}\alpha^{2}_{k}}{\mu} + \frac{128(L+1)^{2}\mathcal{T}^{2}\sigma^{2}f\alpha^{2}_{k}}{|\mathcal{H}|}\notag,
\end{align}
where the last inequality is due to \eqref{eq:filter_condition} 
\begin{align*}
    \frac{23\mu}{12} - \frac{17L|\mathcal{B}_{k}|}{3|\mathcal{H}|} \geq \frac{\mu}{36}\cdot
\end{align*}
Using $\beta_{k} = \frac{72}{\mu\mathcal{T}(1+h+k)}$ we obtain from above
\begin{align*}
    \mathbb{E}[V_{k+1}]&\leq \Big(1-\frac{2}{1+h+k}\Big)\mathbb{E}[V_{k}]\notag\\
    &\quad + \frac{150(L+1)^{3}\mathcal{T}^{2}\sigma^{2}C^{2}_{\alpha}}{\mu(1+h+k)^{2}} + \frac{128(L+1)^{2}\mathcal{T}^{2}\sigma^{2}fC^{2}_{\alpha}}{(1+h+k)^{2}|\mathcal{H}|},
\end{align*}
which by multiplying both sides  by $(1+h+k)^2$ gives
\begin{align*}
    &(1+h+k)^{2}\mathbb{E}[V_{k+1}]\notag\\
    &\leq (h+k)^{2}\mathbb{E}[V_{k}] \notag\\
     &\quad+ \frac{150(L+1)^{3}\mathcal{T}^{2}\sigma^{2}C^{2}_{\alpha}}{\mu} + \frac{128(L+1)^{2}\mathcal{T}^{2}\sigma^{2}fC^{2}_{\alpha}}{|\mathcal{H}|}\notag\\
    &\leq h^{2}\mathbb{E}[V_{0}]+ \frac{150(L+1)^{3}\mathcal{T}^{2}\sigma^{2}C^{2}_{\alpha}(k+1)}{\mu} \notag\\
    &\quad+ \frac{128(L+1)^{2}\mathcal{T}^{2}\sigma^{2}fC^{2}_{\alpha}(k+1)}{|\mathcal{H}|}\cdot
\end{align*}
By dividing both sides of the above inequality by $(1+h+k)^{2}$ we immediately obtain \eqref{eq:convergence_rate_SC}, which concludes our proof. 
\end{proof}
\subsection{Non-convex satisfying PL condition}\label{sec:PL}
In this section, we will present the results for the case where  $q_{\mathcal{H}}(x)$ satisfy the so-called P\L\; condition presented below.
 \begin{assumption}\label{eq:Assumption_PL_condition} There exists a constant $\mu>0$ s.t. 
\begin{align}\label{eq:PL condition}
    \hspace{-.3cm}\frac{1}{2}\|\nabla q_{\mathcal{H}}(\Bar{x}_{k})\|^{2} \geq \mu (q_{\mathcal{H}}(\Bar{x}_{k}) - q_{\mathcal{H}}(x^{\star}_{\mathcal{H}})) \geq \frac{\mu^{2}}{2}\|\Bar{x}_{k} - x^{\star}_{\mathcal{H}}\|^{2}.
\end{align}
\end{assumption}
Next we consider following lemma, where we present its proof in the Appendix.
\begin{lemma}\label{lem:Lemma_V1_PL}
We have for all $k\geq 0$
    \begin{align}\label{lem:Lemma_V1_PL:ineq}
        &\mathbb{E}[ q_{\mathcal{H}}(\Bar{x}_{k+1}) - q_{\mathcal{H}}(\Bar{x}_{k})]\notag\\
        &\leq \Big(-\frac{5\mathcal{T}\beta_{k}}{6} + \frac{2L\mathcal{T}\beta_{k}|\mathcal{B}_{k}|}{|\mathcal{H}|} + \frac{110L^{3}\mathcal{T}^{2}\beta^{2}_{k}}{\mu^{2}}\Big)\mathbb{E}[\|\nabla q_{\mathcal{H}}(\bar{x}_{k})\|^{2}]\notag\\
    & \quad +50\mathcal{T}\beta_{k}\mathbb{E}[W_{k}] + 30\mathcal{T}\sigma^{2}\alpha^{2}_{k} + 16L\mathcal{T}^{2}\sigma^{2}\alpha_{k}\beta^{2}_{k} + \frac{16\mathcal{T}^{2}\sigma^{2}f\alpha^{2}_{k}}{|\mathcal{H}|}\cdot
    \end{align}
\end{lemma}
For our result, we consider the following Lyapunov function 
\begin{align}\label{eq:Lyapunov_PL_def}
    V_{k} = (q_{\mathcal{H}}(\bar{x}_{k})-q_{\mathcal{H}}(x^{\star}_{\mathcal{H}})) + W_{k}.
\end{align}
\begin{thm}\label{Theorem:PL_Tg1}

Let Assumption \ref{eq:Assumption_PL_condition} hold. Let $\alpha_{k}$ and $\beta_{k}$ be given in  \eqref{eq:step_sizes} and satisfy
    \begin{align}\label{eq:alpha_beta_choice_PL}
        \alpha_{k} &\leq \frac{\mu^{2}}{(6)^{4}(L+1)^{3}\mathcal{T}}, \quad \beta_{k} \leq \frac{\mu^{2}}{(12)^{4}L^{3}\mathcal{T}},\notag\\
    &\quad \frac{\beta_{k}}{\alpha_{k}} \leq \frac{\mu^{2}}{(12)^{4}(L+1)^{4}\mathcal{T}^{2}}\cdot
    \end{align}
    Then for all \(k \geq 0\) we have
    \begin{align}\label{eq:Lyapunov_update_PL_thm}
        \mathbb{E}[V_{k+1}]
        &\leq \Big(1 -\frac{9\mu\mathcal{T}\beta_{k}}{6} + \frac{4L\mathcal{T}\beta_{k}|\mathcal{B}_{k}|}{|\mathcal{H}|}\Big)\mathbb{E}[V_{k}]\notag\\
        & \quad+150(L+1)^{2}\mathcal{T}^{2}\sigma^{2}\alpha^{2}_{k} + \frac{112(L+1)^{2}\mathcal{T}^{2}\sigma^{2}f\alpha_{k}^{2}}{|\mathcal{H}|}\cdot
    \end{align}
    Further, let $C_{\alpha}, C_{\beta}$ and $h$ satisfy
    \begin{align}\label{eq:step_size_parameter_PL}
         & C_{\alpha} \geq \frac{(12)^{5}(L+1)^{4}\mathcal{T}}{\mu^{3}}, \quad C_{\beta} = \frac{12}{\mu\mathcal{T}}\notag\\
         & h \geq \max\Big\{\frac{(6)^{4}(L+1)^{3}\mathcal{T}C_{\alpha}}{\mu^{2}};\ \frac{(12)^{5}L^{3}}{\mu^{3}}\Big\},
    \end{align}
    and the condition \eqref{eq:filter_condition} hold. Then we obtain
\begin{align}\label{eq:convergence_rate_PL}
    &\mathbb{E}[V_{k+1}] \notag\\
    &\leq \frac{h^{2}\mathbb{E}[V_{0}]}{(1+h+k)^{2}} + \frac{150\mathcal{T}\sigma^{2}C_{\alpha}^{2}}{(1+h+k)} + \frac{112\mathcal{T}^{2}\sigma^{2}fC_{\alpha}^{2}}{(1+h+k)|\mathcal{H}|}\cdot
\end{align}    
\end{thm}
\begin{remark}
Our result in \eqref{eq:convergence_rate_PL} implies that $q_{\mathcal{H}}(\bar{x}_k)$ converges to the optimal function value $q_{\mathcal{H}}(x^{\star}_{\mathcal{H}})$  at a rate $\mathcal{O}(1/k)$, which is the same rate as in the Byzantine-free setting. In addition, our convergence complexity bound depends on the ratio $f/\mathcal{H}$, which is similar to the result in \cite{dutta2023resilient}. 
\end{remark}
\begin{proof}
    Using Lemmas \ref{lem:Lemma_V2_SC} and \ref{lem:Lemma_V1_PL} we obtain
    \begin{align} 
        &\mathbb{E}[ q_{\mathcal{H}}(\Bar{x}_{k+1})- q_{\mathcal{H}}(\Bar{x}_{k})] + \mathbb{E}[W_{k+1}]\notag\\
         &\leq \Big(-\frac{5\mathcal{T}\beta_{k}}{6} + \frac{2L\mathcal{T}\beta_{k}|\mathcal{B}_{k}|}{\mu|\mathcal{H}|} + \frac{110L^{3}\mathcal{T}^{2}\beta^{2}_{k}}{\mu^{2}}\Big) \mathbb{E}[\|\nabla q_{\mathcal{H}}(\bar{x}_{k})\|^{2}]\notag\\
         &\quad+ \Big(\frac{14L^{3}\mathcal{T}^{2}\alpha_{k}\beta_{k}}{\mu^{2}}+ \frac{300L^{4}\mathcal{T}^{2}\beta^{2}_{k}}{\mu^{2}\alpha_{k}}\Big)\mathbb{E}[\|\nabla q_{\mathcal{H}}(\bar{x}_{k})\|^{2}]\notag\\
         &\quad+\Big(1-\frac{\alpha_{k}}{2} + \frac{78\beta_{k}\mathcal{T}}{\mu} + 126(L+1)^{2}\mathcal{T}^{2}\beta_{k}\Big)\mathbb{E}[W_{k}]\notag\\
         &\quad+30\mathcal{T}\sigma^{2}\alpha^{2}_{k} + 16L\mathcal{T}^{2}\sigma^{2}\alpha_{k}\beta^{2}_{k} + 2\mathcal{T}\sigma^{2}\alpha_{k}^{2}\notag\\
         &\quad+ 96L^{2}\mathcal{T}^{2}\sigma^{2}\beta^{2}_{k} + 6L^{2}\mathcal{T}^{2}\sigma^{2}\alpha^{2}_{k}\beta^{2}_{k}\notag\\
         &\quad+ \frac{16\mathcal{T}^{2}\sigma^{2}f\alpha^{2}_{k}}{|\mathcal{H}|} + \frac{96L^{2}\mathcal{T}^{2}\sigma^{2}f\alpha^{2}_{k}}{|\mathcal{H}|},\notag
    \end{align}
    which since $\beta_{k}\leq \alpha_{k}\leq 1$ 
   gives
    \begin{align}
        &\mathbb{E}[ q_{\mathcal{H}}(\Bar{x}_{k+1})- q_{\mathcal{H}}(x^{\star}_{k})] + \mathbb{E}[W_{k+1}]\notag\\
        &\quad-\mathbb{E}[ q_{\mathcal{H}}(\Bar{x}_{k})- q_{\mathcal{H}}(x^{\star}_{k})]\notag\\
         &\leq \Big(-\frac{5\mathcal{T}\beta_{k}}{6} + \frac{2L\mathcal{T}\beta_{k}|\mathcal{B}_{k}|}{\mu|\mathcal{H}|} + \frac{110L^{3}\mathcal{T}^{2}\beta^{2}_{k}}{\mu^{2}}\Big) \mathbb{E}[\|\nabla q_{\mathcal{H}}(\bar{x}_{k})\|^{2}]\notag\\
         &\quad+ \Big(\frac{14L^{3}\mathcal{T}^{2}\alpha_{k}\beta_{k}}{\mu^{2}}+ \frac{300L^{4}\mathcal{T}^{2}\beta^{2}_{k}}{\mu^{2}\alpha_{k}}\Big)\mathbb{E}[\|\nabla q_{\mathcal{H}}(\bar{x}_{k})\|^{2}]\notag\\
         &\quad+\Big(1-\frac{\alpha_{k}}{2} + \frac{78\beta_{k}\mathcal{T}}{\mu} + 126(L+1)^{2}\mathcal{T}^{2}\beta_{k}\Big)\mathbb{E}[W_{k}]\notag\\
         &\quad+150(L+1)^{2}\mathcal{T}^{2}\sigma^{2}\alpha^{2}_{k} + \frac{112(L+1)^{2}\mathcal{T}^{2}\sigma^{2}f\alpha_{k}^{2}}{|\mathcal{H}|}\notag\cdot
    \end{align}
    By the definition of $V_{k}$ in \eqref{eq:Lyapunov_PL_def} we have from the relation above
    \begin{align*}
        &\mathbb{E}[V_{k+1}]\notag\\
        &\leq \Big(1-\frac{9\mu\mathcal{T}\beta_{k}}{6}+\frac{4L\mathcal{T}\beta_{k}|\mathcal{B}_{k}|}{|\mathcal{H}|}\Big)\mathbb{E}[V_{k}]\notag\\
        &\quad+\Big(\frac{9\mu\mathcal{T}\beta_{k}}{6}-\frac{4L\mathcal{T}\beta_{k}|\mathcal{B}_{k}|}{|\mathcal{H}|}\Big)\mathbb{E}[q_{\mathcal{H}}(\bar{x}_{k})-q_{\mathcal{H}}(x^{\star}_{\mathcal{H}})]\notag\\
        &\quad+\Big(-\frac{5\mathcal{T}\beta_{k}}{6} + \frac{2L\mathcal{T}\beta_{k}|\mathcal{B}_{k}|}{\mu|\mathcal{H}|} + \frac{110L^{3}\mathcal{T}^{2}\beta^{2}_{k}}{\mu^{2}}\Big) \mathbb{E}[\|\nabla q_{\mathcal{H}}(\bar{x}_{k})\|^{2}]\notag\\
         &\quad+ \Big(\frac{14L^{3}\mathcal{T}^{2}\alpha_{k}\beta_{k}}{\mu^{2}}+ \frac{300L^{4}\mathcal{T}^{2}\beta^{2}_{k}}{\mu^{2}\alpha_{k}}\Big)\mathbb{E}[\|\nabla q_{\mathcal{H}}(\bar{x}_{k})\|^{2}]\notag\\
         &\quad+\Big(-\frac{\alpha_{k}}{2} + \frac{78\beta_{k}\mathcal{T}}{\mu} + 126(L+1)^{2}\mathcal{T}^{2}\beta_{k}\Big)\mathbb{E}[W_{k}]\notag\\
         &\quad+\Big(\frac{9\mu\mathcal{T}\beta_{k}}{6}-\frac{4L\mathcal{T}\beta_{k}|\mathcal{B}_{k}|}{|\mathcal{H}|}\Big)\mathbb{E}[W_{k}]\notag\\&\quad+150(L+1)^{2}\mathcal{T}^{2}\sigma^{2}\alpha^{2}_{k} + \frac{112(L+1)^{2}\mathcal{T}^{2}\sigma^{2}f\alpha_{k}^{2}}{|\mathcal{H}|},
    \end{align*}
    which by Assumption \ref{eq:Assumption_PL_condition} gives
    \begin{align*}
        &\mathbb{E}[V_{k+1}]\notag\\
        &\leq \Big(1-\frac{9\mu\mathcal{T}\beta_{k}}{6} + \frac{4L\mathcal{T}\beta_{k}|\mathcal{B}_{k}|}{|\mathcal{H}|}\Big)\mathbb{E}[V_{k}]\notag\\
         &\quad +\Big(-\frac{\mathcal{T}\beta_{k}}{12}+ \frac{110L^{3}\mathcal{T}^{2}\beta^{2}_{k}}{\mu^{2}}\Big)\mathbb{E}[\|\nabla q_{\mathcal{H}}(\bar{x}_{k})\|^{2}]\notag\\
         &\quad+ \Big(\frac{14L^{3}\mathcal{T}^{2}\alpha_{k}\beta_{k}}{\mu^{2}}+ \frac{300L^{4}\mathcal{T}^{2}\beta^{2}_{k}}{\mu^{2}\alpha_{k}}\Big)\mathbb{E}[\|\nabla q_{\mathcal{H}}(\bar{x}_{k})\|^{2}]\notag\\
         &\quad+\Big(-\frac{\alpha_{k}}{2} + \frac{9\mu\mathcal{T}\beta_{k}}{6} + \frac{204(L+1)^{4}\mathcal{T}^{2}\beta_{k}}{\mu^{2}}\Big)\mathbb{E}[W_{k}]\notag\\
         &\quad+150(L+1)^{2}\mathcal{T}^{2}\sigma^{2}\alpha^{2}_{k} + \frac{112(L+1)^{2}\mathcal{T}^{2}\sigma^{2}f\alpha_{k}^{2}}{|\mathcal{H}|}\notag\\
         &\leq \Big(1-\frac{9\mu\mathcal{T}\beta_{k}}{6} + \frac{4L\mathcal{T}\beta_{k}|\mathcal{B}_{k}|}{|\mathcal{H}|}\Big)\mathbb{E}[V_{k}]\notag\\
    &\quad +\Big(-\frac{\mathcal{T}\beta_{k}}{12}+ \frac{110L^{3}\mathcal{T}^{2}\beta^{2}_{k}}{\mu^{2}}\Big)\mathbb{E}[\|\nabla q_{\mathcal{H}}(\bar{x}_{k})\|^{2}]\notag\\
         &\quad+ \Big(\frac{14L^{3}\mathcal{T}^{2}\alpha_{k}\beta_{k}}{\mu^{2}}+ \frac{300L^{4}\mathcal{T}^{2}\beta^{2}_{k}}{\mu^{2}\alpha_{k}}\Big)\mathbb{E}[\|\nabla q_{\mathcal{H}}(\bar{x}_{k})\|^{2}]\notag\\
         &\quad+\Big(-\frac{\alpha_{k}}{2} +\frac{206(L+1)^{4}\mathcal{T}^{2}\beta_{k}}{\mu^{2}}\Big)\mathbb{E}[W_{k}]\notag\\
         &\quad+150(L+1)^{2}\mathcal{T}^{2}\sigma^{2}\alpha^{2}_{k} + \frac{112(L+1)^{2}\mathcal{T}^{2}\sigma^{2}f\alpha_{k}^{2}}{|\mathcal{H}|}\notag\\
         &\leq \Big(1-\frac{9\mu\mathcal{T}\beta_{k}}{6} + \frac{4L\mathcal{T}\beta_{k}|\mathcal{B}_{k}|}{|\mathcal{H}|}\Big)\mathbb{E}[V_{k}]\notag\\
         &\quad+150(L+1)^{2}\mathcal{T}^{2}\sigma^{2}\alpha^{2}_{k} + \frac{112(L+1)^{2}\mathcal{T}^{2}\sigma^{2}f\alpha_{k}^{2}}{|\mathcal{H}|},
    \end{align*}
    where the last inequality we use \eqref{eq:alpha_beta_choice_PL} to have 
\begin{align*}
    &0\leq  \frac{\mathcal{T}\beta_{k}}{12}- \frac{110L^{3}\mathcal{T}^{2}\beta^{2}_{k}}{\mu^{2}} - \frac{14L^{3}\mathcal{T}^{2}\alpha_{k}\beta_{k}}{\mu^{2}} - \frac{300L^{4}\mathcal{T}^{2}\beta^{2}_{k}}{\mu^{2}\alpha_{k}},\notag\\
    &0 \leq \frac{\alpha_{k}}{2} - \frac{206(L+1)^{4}\mathcal{T}^{2}\beta_{k}}{\mu^{2}}\cdot
\end{align*}
To show  \eqref{eq:convergence_rate_PL},  we use \eqref{eq:filter_condition} into the relation above to obtain
     \begin{align*}
        \mathbb{E}[V_{k+1}]&\leq \Big(1-\frac{\mu\beta_{k}\mathcal{T}}{6}\Big)\mathbb{E}[V_{k}]\notag\\
       &\quad+150(L+1)^{2}\mathcal{T}^{2}\sigma^{2}\alpha^{2}_{k} + \frac{112(L+1)^{2}\mathcal{T}^{2}\sigma^{2}f\alpha_{k}^{2}}{|\mathcal{H}|},
    \end{align*}
    which by using $\beta_{k} = \frac{12}{\mu\mathcal{T}(1+h+k)}$  gives
\begin{align*}
    \mathbb{E}[V_{k+1}]&\leq \Big(1-\frac{2}{1+h+k}\Big)\mathbb{E}[V_{k}] \notag\\
       &\quad+ \frac{150\mathcal{T}\sigma^{2}C_{\alpha}^{2}}{(1+h+k)^{2}} + \frac{112\mathcal{T}^{2}\sigma^{2}fC_{\alpha}^{2}}{(1+h+k)^{2}|\mathcal{H}|}.
\end{align*}
Multiplying both sides of the above inequality by $(1+h+k)^2$ gives
\begin{align*}
    &(1+h+k)^{2}\mathbb{E}[V_{k+1}]\notag\\
    &\leq (h+k)^{2}\mathbb{E}[V_{k}] + 40\mathcal{T}\sigma^{2}C_{\alpha}^{2} + \frac{112\mathcal{T}^{2}\sigma^{2}fC_{\alpha}^{2}}{|\mathcal{H}|}\notag\\
    &\leq h^{2}\mathbb{E}[V_{0}]+ 150\mathcal{T}\sigma^{2}C_{\alpha}^{2}(k+1) + \frac{112\mathcal{T}^{2}\sigma^{2}fC_{\alpha}^{2}(k+1)}{|\mathcal{H}|},
\end{align*}
which when diving both sides by $(k + 1 + h)^2$  gives \eqref{eq:convergence_rate_PL}. 


\end{proof}
\section{Simulations}\label{sec:simulations}

    
\begin{figure*}[]
    \centering
    \subfloat[]{\includegraphics[width=0.45\linewidth, height=0.25\textwidth]{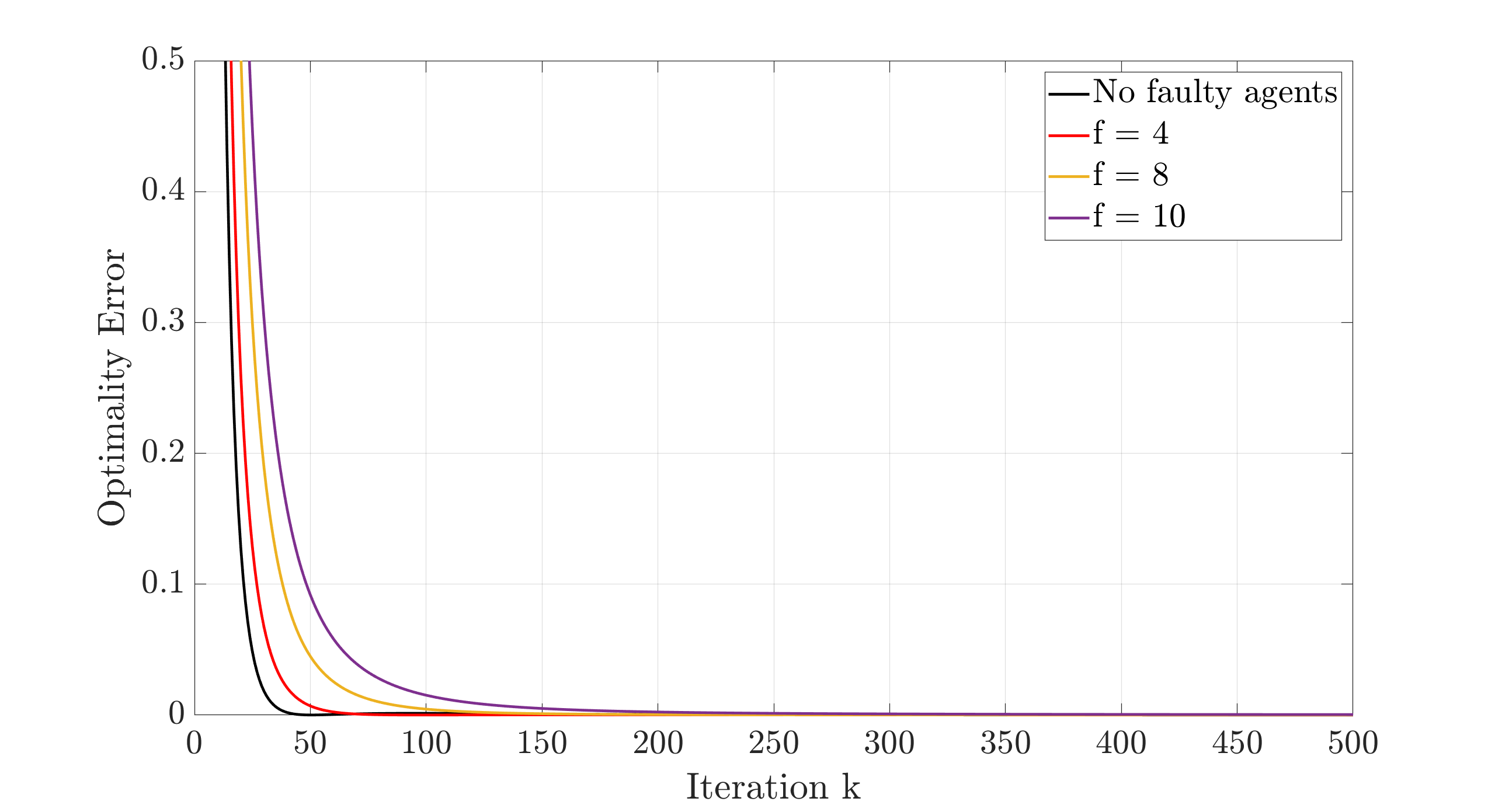}}
    \subfloat[]{\includegraphics[width=0.45\linewidth, height=0.25\textwidth]{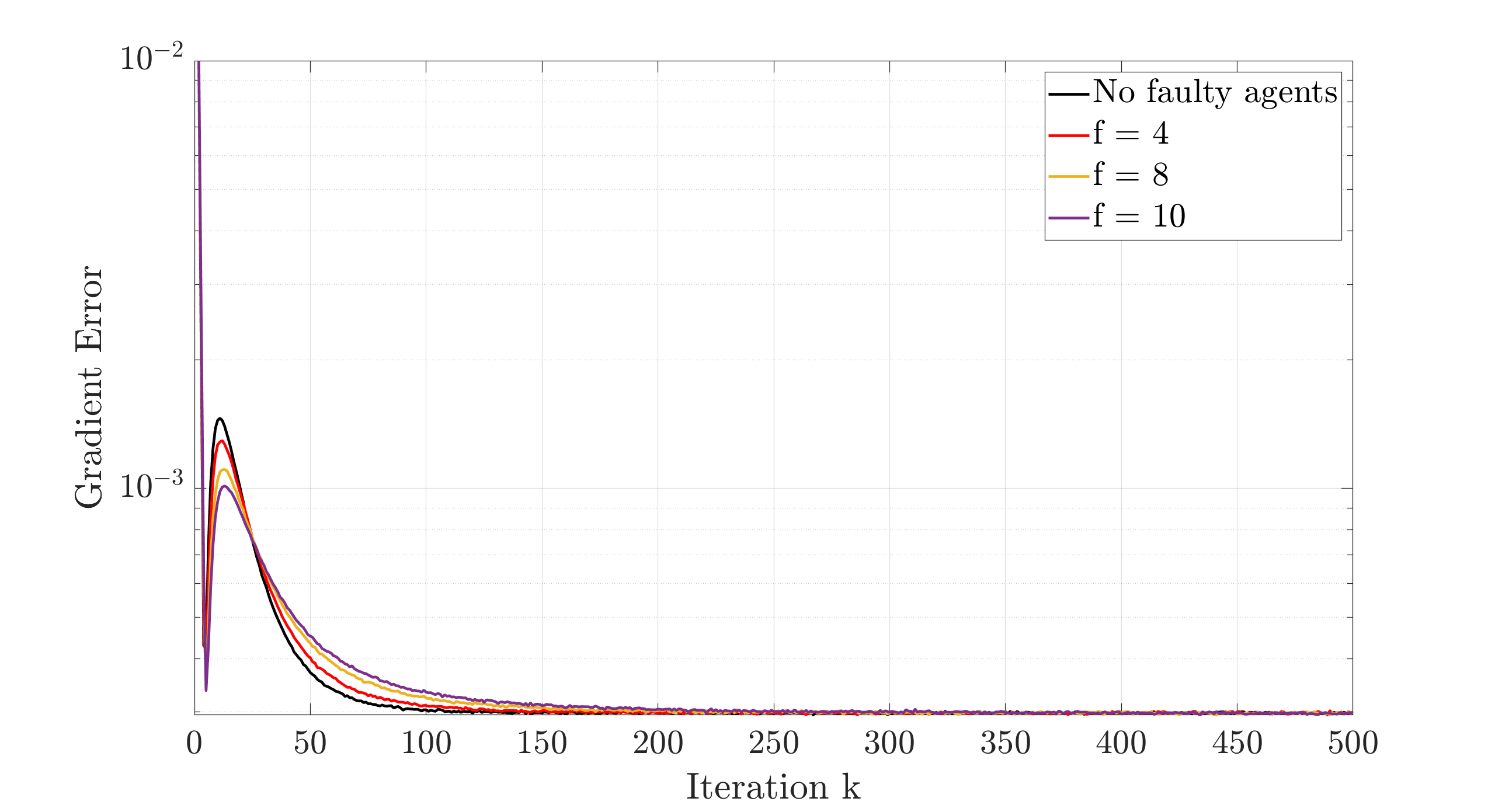}}
    \caption{The convergence of the optimal error, $\|\Bar{x}_{k} - x^{\star}\|^{2}$, and the gradient error, $W_{k}$, for the strongly convex setting.}
    \label{fig:SC_plots}
\end{figure*}

\begin{figure*}[]
    \centering
    \subfloat[]{\includegraphics[width=0.45\linewidth, height=0.25\textwidth]{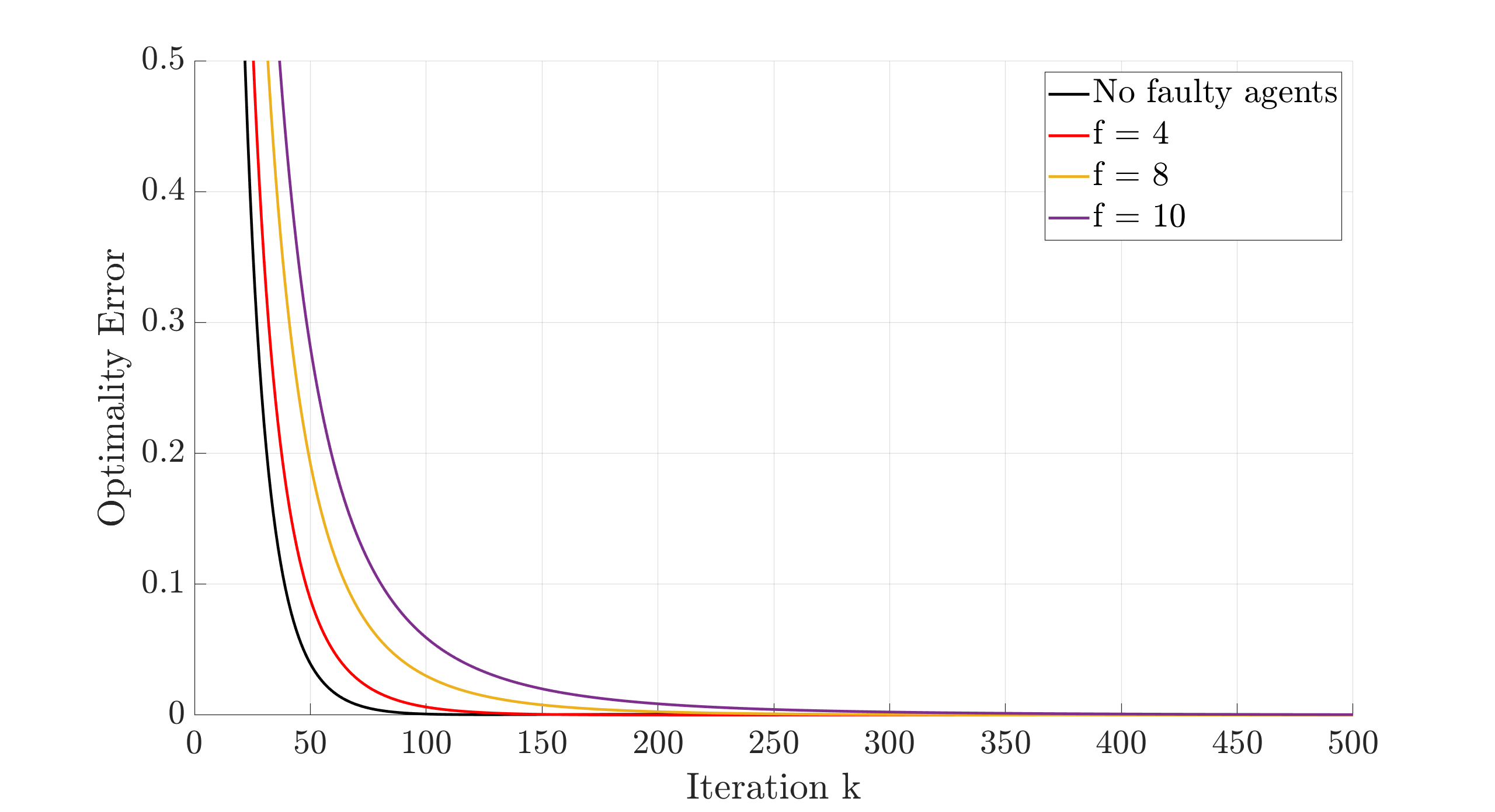}}
    \subfloat[]{\includegraphics[width=0.45\linewidth, height=0.25\textwidth]{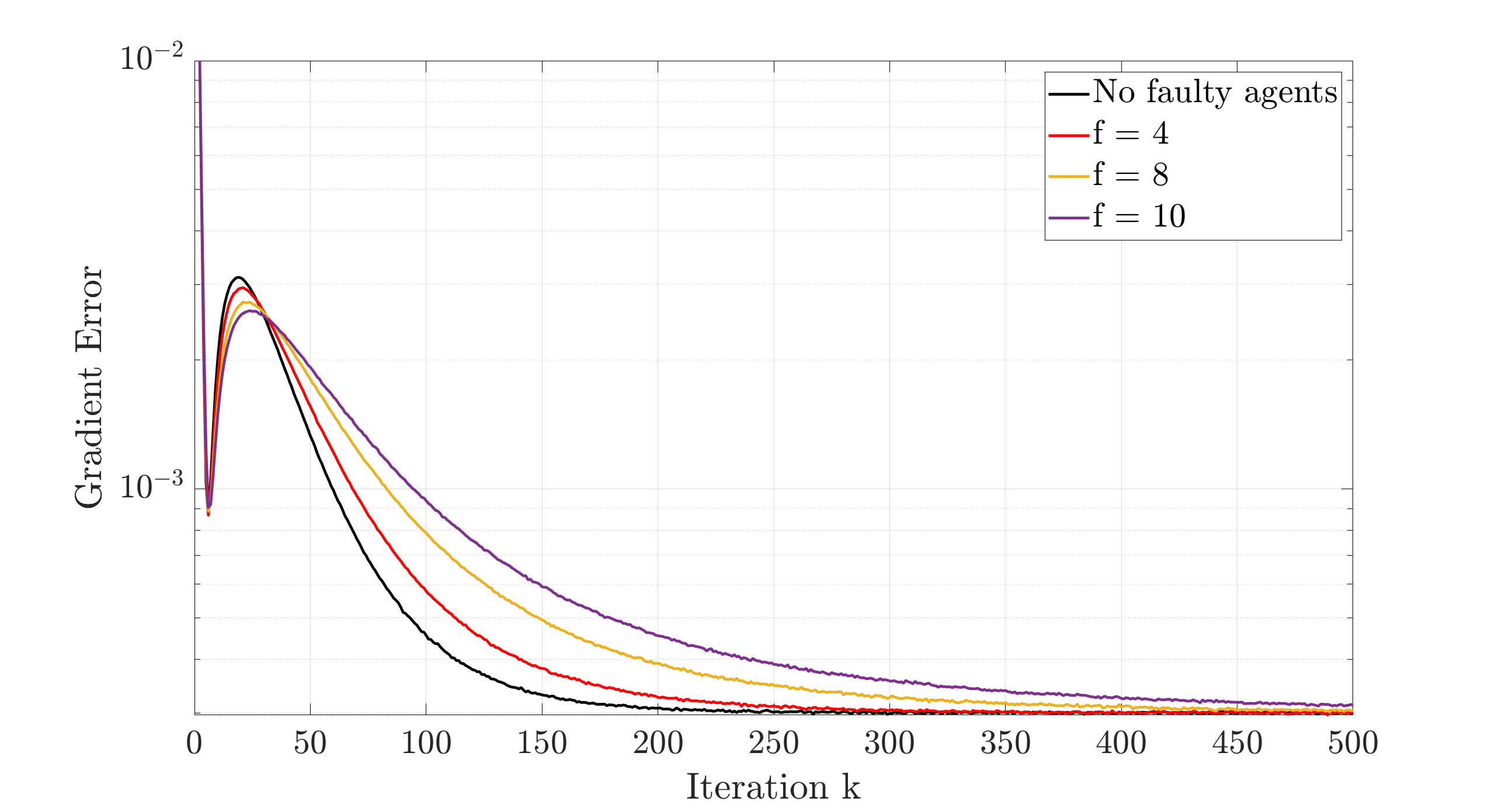}}
    \caption{The convergence of the optimal error, $q_{\mathcal{H}}(\Bar{x}_{k})-q_{\mathcal{H}}(x^{\star}_{\mathcal{H}})$, and the gradient error, $W_{k}$, for the P\L\; condition.}
    \label{fig:PL_plots}
\end{figure*}



In this section we present a few simulations to illustrate the convergence of Algorithm \ref{alg:cap} and the correctness of our theoretical results. For our simulations, we consider a network of $N = 50$ agents. Each non-faulty agent $i$ has access to 100 noisy observations of a 10-dimensional vector ${x}^{\star}$. Specifically, the sample set ${X}^i$ comprises $100$ samples distributed as ${X}^i_j = {x}^\star + {Z}_j$, where \({Z}_j \sim \mathcal{N}(0, I_d)\). On the other hand, a Byzantine faulty agent $j$ mimics the behavior of an honest agent but with different samples. Each sample for a Byzantine agent is given by ${X}^j_j = 2 \times {x}^\star + {Z}_j$, where ${Z}_j \sim \mathcal{N}(0, I_d)$, similar to the honest agents. This implies that while honest agents send information corresponding to Gaussian noisy observations of ${x}^\star$, the Byzantine agents send information corresponding to Gaussian noisy observations with the same variance but centered at $2 \times {x}^\star$.

We will simulate Algorithm \ref{alg:cap} in both strongly convex and P\L\; conditions, where we set the local steps $\mathcal{T} = 3$. In each case, we vary the number of Byzantine agents $f = 4,8,10$ to study the convergence of our algorithm when this number is changing.  For the strongly convex setting, we consider the local cost function of $i^{\text{th}}$ agent as

\begin{align}\label{eq:simulation_problem_SC}
    q^i(x; {X}^i) = \frac{1}{2} \|{x} - {X}^i\|^2.
\end{align}

For the P\L\; condition, we consider the following local cost function

\begin{align}\label{eq:simulation_problem_PL}
    q^i(x; {X}^i) = \frac{1}{2} \|{x} - {X}^i\|^2 + \frac{1}{2} \sin^2(\|{x} - {X}^i\|).
\end{align}
In this case, the global function $\sum_{i = 1}^{50} q^i(x; X^i)$ represents a non-convex function that satisfies the Polyak-Łojasiewicz (PL) condition. 




Our simulation results are shown in Figs. \ref{fig:SC_plots} and \ref{fig:PL_plots} for strongly convex and P\L\; conditions, respectively. 
First, our simulations show that the optimization and gradient estimate errors converge to zero as expected.  Second, the rates of convergence seem to be $\mathcal{O}(1/k)$, which are consistent in both cases. Finally, the algorithm converges slower as the number of faulty agents increases, agreeing with our theoretical bounds in Theorems \ref{Theorem:SC_Tg1} and \ref{Theorem:PL_Tg1}.

\section{Conclusion}
In this paper, we propose a new two-time-scale variant of the local SGD method to solve an exact Byzantine fault-tolerance problem under the $2f$-redundancy condition. 
Our theoretical analysis demonstrates that our approach effectively mitigates the impact of noise from stochastic gradients and the interference of Byzantine agents. Notably, our algorithm achieves an optimal rate $\mathcal{O}(1/k)$ when the underlying objective function satisfies either strong convexity or the P\L\; condition, similar to that of the Byzantine-free setting.

\bibliographystyle{IEEEtran}
\bibliography{IEEEfull,references}

\appendix[Proofs of Lemmas \ref{lem:Lemma_V2_SC}--\ref{lem:Lemma_V1_PL}]
We now proceed to present the analysis of Lemmas \ref{lem:Lemma_V2_SC}--\ref{lem:Lemma_V1_PL}. First, we rewrite \eqref{eq:x_bar_update_0} as follows
\begin{align}\label{eq:x_bar_update}
    \bar{x}_{k+1}&= \bar{x}_{k} - \mathcal{T}\beta_{k}\nabla q_{\mathcal{H}}(\bar{x}_{k}) - \frac{\beta_{k}}{|\mathcal{H}|}\sum_{i \in \mathcal{H}}\sum_{t=0}^{\mathcal{T}-1}e^{i}_{k,t} + \mathcal{E}_{x}.
\end{align}
where 
    \begin{align}\label{eq:V_x}
        \mathcal{E}_{x} = \frac{1}{|\mathcal{H}|}\Big[\sum_{i\in \mathcal{B}_{k}}(x^{i}_{k,\mathcal{T}} -\Bar{x}_{k}) - \sum_{i\in \mathcal{H} \backslash \mathcal{H}_{k}}(x^{i}_{k,\mathcal{T}}-\Bar{x}_{k})\Big].
    \end{align}
Next, we will consider the following results that will be used in our analysis later. 
\begin{lemma}\label{lem:x_k_t_minus_x_bar_squared}
For each honest client $i \in \mathcal{H}$ we have
    \begin{align}\label{lem:x_k_t_minus_x_bar_squared:ineq}
        \hspace{-0.22cm}\|x^{i}_{k,t+1}-\Bar{x}_{k}\|^{2} \leq 2L^{2}t^{2}\beta_{k}^{2}\|\Bar{x}_{k}-x^{\star}_{\mathcal{H}}\|^{2} + 2t\beta_{k}^{2}\sum_{l=0}^{t}\|e^{i}_{k,l}\|^{2}.
    \end{align}
\end{lemma}
\begin{proof}
    By \eqref{alg:xi} and \eqref{notation:e_i}, we have 
\begin{align} 
    &\|x^{i}_{k,t+1}-\Bar{x}_{k}\|^{2} = \Big\|\sum_{l=0}^{t}(x^{i}_{k,l+1}-x^{i}_{k,l})\Big\|^{2}\notag\\
    & \leq t\sum_{l=0}^{t}\|x^{i}_{k,l+1}-x^{i}_{k,l}\|^{2} = \beta_{k}^{2}t\sum_{l=0}^{t}\|y^{i}_{k,l}\|^{2}\notag\\
    &= \beta_{k}^{2}t\sum_{l=0}^{t}\|e_{k,t}^{i} + \nabla q^{i}(\bar{x}_{k})\|^{2}\notag\\
    &\leq 2\beta_{k}^{2}t\sum_{l=0}^{t}\big(\|e_{k,t}^{i}\|^2  + \|\nabla q^{i}(\bar{x}_{k})\|^{2}\big)\notag\\
    &= 2\beta_{k}^{2}t\sum_{l=0}^{t}\big(\|e_{k,t}^{i}\|^2  + \|\nabla q^{i}(\bar{x}_{k} -\nabla q^{i}(x_{\mathcal{H}}^{\star})\|^{2}\big)\notag\\
    &\leq 2L^{2}t^{2}\beta_{k}^{2}\|\Bar{x}_{k}-x^{\star}_{\mathcal{H}}\|^{2} + 2t\beta_{k}^{2}\sum_{l=0}^{t}\|e^{i}_{k,l}\|^{2},\notag
\end{align}
where in the last equality we use \eqref{eq:redundancy_condition} to have $\nabla q^{i}(x_{\mathcal{H}}^{\star}) = 0$ and the last inequality is due to Assumption \ref{eq:Assumption_Lipschitz}. This concludes our proof. 
\end{proof}
\begin{lemma}\label{lem:e_k_t}
For each honest client $i \in \mathcal{H}$ we have 
\begin{align}\label{lem:e_k_t:eq}
&\mathbb{E}[\|e_{k,t+1}^{i}\|^{2}]\notag\\
&\leq (1-\alpha_{k})\mathbb{E}[\|e^{i}_{k,t}\|^{2}] + \alpha_{k}L^{2}\mathbb{E}[\|x^{i}_{k,t}-\bar{x}_{k}\|^{2}] + \alpha^{2}_{k}\sigma^{2}. 
\end{align}    
\end{lemma}

\begin{proof}
    Using \eqref{notation:e_i} and \eqref{alg:yi} we consider
\begin{align*}
    &e^{i}_{k,t+1} = (1-\alpha_{k})y^{i}_{k,t} + \alpha_{k}\nabla q^{i}(x^{i}_{k,t};\Delta^{i}_{k,t}) - \nabla q^{i}(\Bar{x}_{k})\notag\\
    & = (1-\alpha_{k})e^{i}_{k,t} + \alpha_{k}(\nabla q^{i}(x^{i}_{k,t};\Delta^{i}_{k,t}) - \nabla q^{i}(x^{i}_{k,t}))\notag\\
    &\quad +\alpha_{k}(\nabla q^{i}(x^{i}_{k,t})-\nabla q^{i}(\Bar{x}_{k})),
\end{align*}
which by using Assumptions \ref{eq:Assumption_noise} and \ref{eq:Assumption_Lipschitz} gives
\begin{align*}
    &\mathbb{E}[\|e^{i}_{k,t+1}\|^{2}|\mathcal{P}_{k,t}]\notag\\
    &= (1-\alpha_{k})^{2}\|e^{i}_{k,t}\|^{2} + \alpha^{2}_{k}\|\nabla q^{i}(x^{i}_{k,t})-\nabla q^{i}(\Bar{x}_{k})\|^{2}\notag\\
    &\quad + \alpha^{2}_{k}\mathbb{E}[\|\nabla q^{i}(x^{i}_{k,t};\Delta^{i}_{k,t}) - \nabla q^{i}(x^{i}_{k,t})\|^{2}|\mathcal{P}_{k,t}]\notag\\
    &\quad + 2\alpha_{k}(1-\alpha_{k})(\nabla q^{i}(x^{i}_{k,t})-\nabla q^{i}(\Bar{x}_{k}))^{T}e^{i}_{k,t}\notag\\
    &\quad + 2 \alpha_{k}(1-\alpha_{k})\mathbb{E}[(\nabla q^{i}(x^{i}_{k,t},\Delta^{i}_{k,t}) - \nabla q^{i}(x^{i}_{k,t}))|\mathcal{P}_{k,t}]^{T}e^{i}_{k,t}\notag\\
    &\quad +2\alpha^{2}_{k}\mathbb{E}\big[(\nabla q^{i}(x^{i}_{k,t},\Delta^{i}_{k,t}) - \nabla q^{i}(x^{i}_{k,t}))|\mathcal{P}_{k,t}\big]^{T}\notag\\
    &\hspace{1.5cm}\times (\nabla q^{i}(x^{i}_{k,t})-\nabla q^{i}(\Bar{x}_{k})\notag\\
    &= (1-\alpha_{k})^{2}\|e^{i}_{k,t}\|^{2} + \alpha^{2}_{k}\|\nabla q^{i}(x^{i}_{k,t})-\nabla q^{i}(\Bar{x}_{k})\|^{2}\notag\\
    &\quad + \alpha^{2}_{k}\mathbb{E}[\|\nabla q^{i}(x^{i}_{k,t};\Delta^{i}_{k,t}) - \nabla q^{i}(x^{i}_{k,t})\|^{2}|\mathcal{P}_{k,t}]\notag\\
    &\quad + 2\alpha_{k}(1-\alpha_{k})(\nabla q^{i}(x^{i}_{k,t})-\nabla q^{i}(\Bar{x}_{k}))^{T}e^{i}_{k,t}\notag\\
    &\leq (1-\alpha_{k})^{2}\|e^{i}_{k,t}\|^{2} + L^2\alpha^{2}_{k}\|x^{i}_{k,t} - \bar{x}_{k}\|^{2} + \alpha^{2}_{k}\sigma^{2}\notag\\
    &\quad  + 2\alpha_{k}(1-\alpha_{k})(\nabla q^{i}(x^{i}_{k,t}) - \nabla q^{i}(\Bar{x}_{k}))^{T}e^{i}_{k,t}.
\end{align*}
Taking the expectation on both sides of the preceding equation and using the Cauchy-Schwarz inequality we obtain \eqref{lem:e_k_t:eq}, i.e.,
\begin{align} 
    &\mathbb{E}[\|e^{i}_{k,t+1}\|^{2}] \notag\\
    &\leq (1-\alpha_{k})^{2}\mathbb{E}[\|e^{i}_{k,t}\|^{2}] + L^{2}\alpha^{2}_{k}\mathbb{E}[\|x^{i}_{k,t} - \bar{x}_{k}\|^{2}\|^{2}] + \alpha^{2}_{k}\sigma^{2}\notag\\
    &\quad  + \alpha_{k}(1-\alpha_{k})\mathbb{E}[\|e^{i}_{k,t}\|^{2}]+ L^{2}\alpha_{k}(1-\alpha_{k})\mathbb{E}[\|x^{i}_{k,t}-\Bar{x}_{k}\|^{2}]\notag\\
    & = (1-\alpha_{k})\mathbb{E}[\|e^{i}_{k,t}\|^{2}] + L^{2}\alpha_{k}\mathbb{E}[\|x^{i}_{k,t} - \bar{x}_{k}\|^{2}] + \alpha^{2}_{k}\sigma^{2}\notag.
\end{align}
\end{proof}
\vspace{-1cm}
\begin{lemma}\label{lem:e_k_t_V12}
Let $\alpha_{k}$ satisfy for all $k\geq 0$
\begin{align}\label{eq:alpha_choice_0}
    \alpha_{k} \leq \frac{1}{2L\mathcal{T}}\cdot
\end{align}
Then the following holds
    \begin{align}\label{lem:e_k_t_V12:ineq}
        &\frac{1}{|\mathcal{H}|}\sum_{i \in \mathcal{H}}\sum_{t=0}^{\mathcal{T}-1}\mathbb{E}[\|e^{i}_{k,t}\|^{2}] \notag\\
        &\leq 2\mathcal{T}\mathbb{E}[W_{k}] + 2\mathcal{T}\sigma^{2}\alpha_{k} + 4L^4\mathcal{T}^3\beta_{k}^2\mathbb{E}[\|\bar{x}_{k}-x_{\mathcal{H}}^{\star}\|^2].
    \end{align}
\end{lemma}
\begin{proof}
    By taking \eqref{lem:e_k_t:eq} recursively and since $\alpha_{k} < 1$ we have
    \begin{align*}
        \mathbb{E}[\|e^{i}_{k,t}\|^{2}]
        &\leq (1-\alpha_{k})^{t}\mathbb{E}[\|e^{i}_{k,0}\|^{2}] + \alpha_{k}^2\sigma^{2}\sum_{l=0}^{t-1}(1-\alpha_{k})^{t-1-\ell}\notag\\
        &\quad  + \alpha_{k}L^{2}\sum_{\ell = 0}^{t-1}(1-\alpha_{k})^{t-l-1}\mathbb{E}[\|x^{i}_{k,l}-\bar{x}_{k}\|^{2}]\notag\\
        &\leq (1-\alpha_{k})^{t}\mathbb{E}[\|e^{i}_{k,0}\|^{2}] + \alpha_{k}\sigma^{2}\notag\\
        &\quad+ \alpha_{k}L^{2}\sum_{\ell = 0}^{t-1}(1-\alpha_{k})^{t-l-1}\mathbb{E}[\|x^{i}_{k,l}-\bar{x}_{k}\|^{2}],
    \end{align*}
    which by using \eqref{lem:x_k_t_minus_x_bar_squared:ineq} gives
    \begin{align*}
        &\mathbb{E}[\|e^{i}_{k,t}\|^{2}]\notag\\
        &\leq (1-\alpha_{k})^{t}\mathbb{E}[\|e^{i}_{k,0}\|^{2}] + \alpha_{k}\sigma^{2}\notag\\
        &\quad +2L^4\mathcal{T}^2\alpha_{k}\beta_{k}^2\sum_{\ell=0}^{t-1}(1-\alpha_{k})^{t-\ell-1}\mathbb{E}[\|\bar{x}_{k}-x_{\mathcal{H}}^{\star}\|^2]\notag\\
        &\quad +2L^2\mathcal{T}\alpha_{k}\beta_{k}^2\sum_{\ell=0}^{t-1}(1-\alpha_{k})^{t-\ell-1}\sum_{m = 0}^{\ell-1}\mathbb{E}[\|e_{k,m}^{i}\|^2]\notag\\
        &\leq (1-\alpha_{k})^{t}\mathbb{E}[\|e^{i}_{k,0}\|^{2}] + \alpha_{k}\sigma^{2}+2L^4\mathcal{T}^2\beta_{k}^2\mathbb{E}[\|\bar{x}_{k}-x_{\mathcal{H}}^{\star}\|^2]\notag\\
        &\quad +2L^2\mathcal{T}\alpha_{k}\beta_{k}^2\sum_{\ell=0}^{t-1}(1-\alpha_{k})^{t-\ell-1}\sum_{m = 0}^{\mathcal{T}-1}\mathbb{E}[\|e_{k,m}^{i}\|^2]\notag\\
        &\leq (1-\alpha_{k})^{t}\mathbb{E}[\|e^{i}_{k,0}\|^{2}] + \alpha_{k}\sigma^{2} + 2L^4\mathcal{T}^2\beta_{k}^2\mathbb{E}[\|\bar{x}_{k}-x_{\mathcal{H}}^{\star}\|^2]\notag\\
        &\quad +2L^2\mathcal{T}\beta_{k}^2\sum_{m = 0}^{\mathcal{T}-1}\mathbb{E}[\|e_{k,m}^{i}\|^2].
    \end{align*}
Using the relation above and the definition of $W_k$ in \eqref{eq:V_2_sc} we have
    \begin{align*}
        &\frac{1}{|\mathcal{H}|}\sum_{i \in \mathcal{H}}\sum_{t=0}^{\mathcal{T}-1}\mathbb{E}[\|e^{i}_{k,t}\|^{2}] \notag\\
        &\leq \frac{1}{|\mathcal{H}|}\sum_{i \in \mathcal{H}}\sum_{t=0}^{\mathcal{T}-1}(1-\alpha_{k})^{t}\mathbb{E}[\|e^{i}_{k,0}\|^{2}] + \mathcal{T}\sigma^{2}\alpha_{k}\notag\\
        &\quad + 2L^4\mathcal{T}^3\beta_{k}^2\mathbb{E}[\|\bar{x}_{k}-x_{\mathcal{H}}^{\star}\|^2] \notag\\
        &\quad + 2L^2\mathcal{T}^2\beta_{k}^2\frac{1}{|\mathcal{H}|}\sum_{i \in \mathcal{H}}\sum_{m = 0}^{\mathcal{T}-1}\mathbb{E}[\|e_{k,m}^{i}\|^2]\notag\\
        &\leq\mathbb{E}[W_{k}] + \mathcal{T}\sigma^{2}\alpha_{k} + 2L^4\mathcal{T}^3\beta_{k}^2\mathbb{E}[\|\bar{x}_{k}-x_{\mathcal{H}}^{\star}\|^2] \notag\\
        &\quad + 2L^2\mathcal{T}^2\beta_{k}^2\frac{1}{|\mathcal{H}|}\sum_{i \in \mathcal{H}}\sum_{m = 0}^{\mathcal{T}-1}\mathbb{E}[\|e_{k,m}^{i}\|^2],
    \end{align*}
    where the last inequality we use the fact that $1-\alpha_{k} \leq 1$.  Since $\beta_{k}\leq\alpha_{k} \leq 1/2L\mathcal{T}$ implying $1-2L^2T^2\beta_{k}^2 > 1/2$, rearranging the preceding relation we obtain \eqref{lem:e_k_t_V12:ineq}.
    

\end{proof}
\begin{lemma}\label{lem:xbar_k_k+1}
Let $\alpha_k$ satisfy \eqref{eq:alpha_choice_0}. Then we have
\begin{align}
    \mathbb{E}[\|\Bar{x}_{k+1} - \Bar{x}_{k}\|^{2}]&\leq 100L^{2}\mathcal{T}^{2}\beta^{2}_{k}\mathbb{E}[\|\bar{x}_{k}-x^{\star}_{\mathcal{H}}\|^{2}]  + 40\beta^{2}_{k}\mathcal{T}^{2}\mathbb{E}[W_{k}] \notag\\
    &\quad+ 32\mathcal{T}^{2}\sigma^{2}\alpha_{k}\beta^{2}_{k} + \frac{32\mathcal{T}^{2}\sigma^{2}f\alpha^{2}_{k}\beta^{2}_{k}}{|\mathcal{H}|}\cdot\label{lem:xbar_k_k+1:ineq}
\end{align}
\end{lemma}

\begin{proof}
    By \eqref{eq:x_bar_update} and since $\nabla q_{\mathcal{H}}(x^{\star}_{\mathcal{H}})) = 0$ we have
\begin{align}
    &\|\Bar{x}_{k+1}-\Bar{x}_{k}\|^{2}\notag\\
    &= \Big\|\mathcal{E}_{k}-\mathcal{T}\beta_{k}\nabla q_{\mathcal{H}}(\Bar{x}_{k})-\frac{\beta_{k}}{|\mathcal{H}|}\sum_{i\in \mathcal{H}}\sum_{l=0}^{\mathcal{T}-1}e^{i}_{k,l}\Big\|^{2}\notag\\
    & \leq 2\|\mathcal{E}_{x}\|^{2} + 2\Big\|\mathcal{T}\beta_{k}(\nabla q_{\mathcal{H}}(\Bar{x}_{k})+\frac{\beta_{k}}{|\mathcal{H}|}\sum_{i\in \mathcal{H}}\sum_{l=0}^{\mathcal{T}-1}e^{i}_{k,l}\Big\|^{2}\notag\\
    & \leq 2\|\mathcal{E}_{x}\|^{2} + 4\mathcal{T}^2\beta_k^2\|\nabla q_{\mathcal{H}}(\Bar{x}_{k})\|^2  + \frac{4\beta^{2}_{k}\mathcal{T}}{|\mathcal{H}|}\sum_{i \in \mathcal{H}}\sum_{l=0}^{\mathcal{T}-1}\|e^{i}_{k,l}\|^{2}\notag\\
    & = 2\|\mathcal{E}_{x}\|^{2} + 4\mathcal{T}^2\beta_k^2\|\nabla q_{\mathcal{H}}(\Bar{x}_{k})-\nabla q_{\mathcal{H}}(x^{\star}_{\mathcal{H}}))\|^2\notag\\  
    &\qquad + \frac{4\beta^{2}_{k}\mathcal{T}}{|\mathcal{H}|}\sum_{i \in \mathcal{H}}\sum_{l=0}^{\mathcal{T}-1}\|e^{i}_{k,l}\|^{2}\notag\\
    &\leq 2\|\mathcal{E}_{x}\|^{2} + 4L^{2}\mathcal{T}^{2}\beta^{2}_{k}\|\Bar{x}_{k} - x^{\star}_{\mathcal{H}}\|^{2} + \frac{4\beta^{2}_{k}\mathcal{T}}{|\mathcal{H}|}\sum_{i \in \mathcal{H}}\sum_{l=0}^{\mathcal{T}-1}\|e^{i}_{k,l}\|^{2}.\notag
\end{align}
Taking the expectation of the preceding relation and using \eqref{lem:e_k_t_V12:ineq} give
\begin{align}\label{eq:x_bar_diff_0}
    &\mathbb{E}[\|\bar{x}_{k+1}-\bar{x}_{k}\|^{2}]\notag\\
    &\leq 2\mathbb{E}[\|\mathcal{E}_{k}\|^{2}] + \big(4L^{2}\mathcal{T}^{2}\beta^{2}_{k} + 16L^{4}\mathcal{T}^{4}\beta^{4}_{k}\big)\mathbb{E}[\|\bar{x}_{k}-x^{\star}_{\mathcal{H}}\|^{2}]\notag\\
    &\quad + 8\mathcal{T}^{2}\beta^{2}_{k}\mathbb{E}[W_{k}].
\end{align}
Next, we analyze the term $\|\mathcal{E}_{x}\|^{2}$. For this using \eqref{eq:V_x}, we have 
     \begin{align*}
         \|\mathcal{E}_{x}\|^{2} = \Big\|\frac{1}{|\mathcal{H}|}\Big[\sum_{i\in \mathcal{B}_{k}}(x^{i}_{k,\mathcal{T}} -\Bar{x}_{k}) - \sum_{i\in \mathcal{H} \backslash \mathcal{H}_{k}}(x^{i}_{k,\mathcal{T}}-\Bar{x}_{k})\Big]\Big\|^{2}.
     \end{align*}
     By \eqref{alg:sort_distances}, we have $\|x^{i}_{k,\mathcal{T}}-\Bar{x}_{k}\| \leq \|x^{j}_{k,\mathcal{T}}-\Bar{x}_{k}\|$ for all $i \in \mathcal{B}_{k}$ and $j \in \mathcal{H}\backslash \mathcal{H}_{k}$. Thus, we obtain
\begin{align} 
    \|\mathcal{E}_{x}\|^{2}&\leq \frac{2|\mathcal{B}_{k}|}{|\mathcal{H}|^{2}}\sum_{i \in \mathcal{B}_{k}}\|x^{i}_{k,\mathcal{T}} - \Bar{x}_{k}\|^{2} + \frac{2|\mathcal{B}_{k}|}{|\mathcal{H}|^{2}}\sum_{i \in \mathcal{H}\backslash\mathcal{H}_{k}}\|x^{i}_{k,\mathcal{T}} - \Bar{x}_{k}\|^{2} \notag\\
    &\leq \frac{4|\mathcal{B}_{k}|}{|\mathcal{H}|^{2}}\sum_{i \in \mathcal{H}\backslash\mathcal{H}_{k}}\|x^{i}_{k,\mathcal{T}} - \Bar{x}_{k}\|^{2},\notag
\end{align}
which by \eqref{lem:x_k_t_minus_x_bar_squared:ineq} yields
\begin{align}
     &\|\mathcal{E}_{x}\|^{2}\notag\\
     & \leq  \frac{4|\mathcal{B}_{k}|}{|\mathcal{H}|^{2}}\sum_{i \in \mathcal{H}\backslash\mathcal{H}_{k}}(2L^{2}\mathcal{T}^{2}\beta_{k}^{2}\|\Bar{x}_{k}-x^{\star}_{\mathcal{H}}\|^{2}+ 2\mathcal{T}\beta_{k}^{2}\sum_{l=0}^{\mathcal{T}-1}\|e^{i}_{k,l}\|^{2})\notag\\
    & \leq \frac{8L^{2}\mathcal{T}^{2}\beta_{k}^{2}|\mathcal{B}_{k}|^{2}}{|\mathcal{H}|^{2}}\|\Bar{x}_{k}-x^{\star}_{\mathcal{H}}\|^{2} + \frac{8\mathcal{T}\beta_{k}^{2}|\mathcal{B}_{k}|}{|\mathcal{H}|^{2}}\sum_{i \in \mathcal{H}\backslash\mathcal{H}_{k}}\sum_{l=0}^{\mathcal{T}-1}\|e^{i}_{k,l}\|^{2}.\notag
\end{align}
Using \eqref{lem:e_k_t_V12:ineq} we obtain from the preceding relation
\begin{align}\label{eq:E_x_last}
    &\mathbb{E}[\|\mathcal{E}_{x}\|^{2}]\notag\\
    &\leq \big(\frac{8L^{2}\mathcal{T}^{2}\beta^{2}_{k}|\mathcal{B}_{k}|^{2}}{|\mathcal{H}|^{2}} + \frac{32\alpha_{k}\beta^{4}_{k}L^{4}\mathcal{T}^{4}|\mathcal{B}_{k}|}{|\mathcal{H}|}\big)\mathbb{E}[\|\bar{x}_{k}-x^{\star}_{\mathcal{H}}\|^{2}]\notag\\
     &\quad +\frac{16\beta^{2}_{k}\mathcal{T}^{2}|\mathcal{B}_{k}|}{|\mathcal{H}|}\mathbb{E}[W_{k}] + \frac{16\alpha^{2}_{k}\beta^{2}_{k}\mathcal{T}^{2}|\mathcal{B}_{k}|\sigma^{2}}{|\mathcal{H}|},
\end{align}
which when using \eqref{eq:x_bar_diff_0} and  $\beta_{k}L\mathcal{T} \leq 1$ gives \eqref{lem:xbar_k_k+1:ineq}, i.e., 
\begin{align*}
    &\mathbb{E}[\|\bar{x}_{k+1}-\bar{x}_{k}\|^{2}]\notag\\
    &\leq \big(\frac{16L^{2}\mathcal{T}^{2}\beta^{2}_{k}|\mathcal{B}_{k}|^{2}}{|\mathcal{H}|^{2}} + \frac{64L^{2}\mathcal{T}^{2}\beta^{2}_{k}|\mathcal{B}_{k}|}{|\mathcal{H}|}\big)\mathbb{E}[\|\bar{x}_{k}-x^{\star}_{\mathcal{H}}\|^{2}]\notag\\
    &\quad + \big(4L^{2}\mathcal{T}^{2}\beta^{2}_{k} + 16L^{4}\mathcal{T}^{4}\beta^{4}_{k}\big)\mathbb{E}[\|\bar{x}_{k}-x^{\star}_{\mathcal{H}}\|^{2}]\notag\\
    &\quad+\big(8\mathcal{T}^{2}\beta^{2}_{k} + \frac{32\mathcal{T}^{2}\beta^{2}_{k}|\mathcal{B}_{k}|}{|\mathcal{H}|}\big)\mathbb{E}[W_{k}]\notag\\
    &\quad + 32\mathcal{T}^{2}\sigma^{2}\alpha_{k}\beta^{2}_{k} + \frac{32\mathcal{T}^{2}\sigma^{2}|\mathcal{B}_{k}|\alpha^{2}_{k}\beta^{2}_{k}}{|\mathcal{H}|}
    \notag\\
    &\leq 100L^{2}\mathcal{T}^{2}\beta^{2}_{k}\mathbb{E}[\|\bar{x}_{k}-x^{\star}_{\mathcal{H}}\|^{2}]  + 40\beta^{2}_{k}\mathcal{T}^{2}\mathbb{E}[W_{k}] \notag\\
    &\quad+ 32\mathcal{T}^{2}\sigma^{2}\alpha_{k}\beta^{2}_{k} + \frac{32\mathcal{T}^{2}\sigma^{2}f\alpha^{2}_{k}\beta^{2}_{k}}{|\mathcal{H}|},
\end{align*}
where the last inequality we use $|\mathcal{B}_{k}| \leq f$, $\frac{|\mathcal{B}_{k}|}{|\mathcal{H}|} \leq \frac{f}{N-f} \leq 1.$

\end{proof}
\subsection*{Proof of Lemma \ref{lem:Lemma_V2_SC}}
From \eqref{notation:e_i} using $y^{i}_{k+1,0} = y^{i}_{k,\mathcal{T}}$ we have
\begin{align*}
    e^{i}_{k+1,0} 
    & = y^{i}_{k,\mathcal{T}}- \nabla q^{i}(\Bar{x}_{k+1})
    = e^{i}_{k,\mathcal{T}} + \nabla q^{i}(\Bar{x}_{k})-\nabla q^{i}(\Bar{x}_{k+1}).
\end{align*}
Using the Cauchy-Schwarz inequality and Assumption \ref{eq:Assumption_Lipschitz} we obtain
\begin{align*}
    \|e^{i}_{k+1,0}\|^{2} 
    &\leq \big(1+\frac{\alpha_{2}}{2}\big)\|e^{i}_{k,\mathcal{T}}\|^{2} + \big(1+\frac{2}{\alpha_{k}}\big)L^{2}\|\Bar{x}_{k+1}-\Bar{x}_{k}\|^{2}\notag\\
    &\leq \big(1+\frac{\alpha_{2}}{2}\big)\|e^{i}_{k,\mathcal{T}}\|^{2} + \frac{3L^{2}}{\alpha_{k}}\|\Bar{x}_{k+1}-\Bar{x}_{k}\|^{2}. 
\end{align*}
Thus, we have
\begin{align}\label{eq:V_2_r0}
    &\mathbb{E}[W_{k+1}] = \frac{1}{|\mathcal{H}|}\sum_{i \in \mathcal{H}}\mathbb{E}[\|e^{i}_{k+1,0}\|^{2}]\notag\\
    &\leq \big(1+\frac{\alpha_{k}}{2}\big)\frac{1}{|\mathcal{H}|}\sum_{i \in \mathcal{H}}\mathbb{E}[\|e^{i}_{k,\mathcal{T}}\|^{2}] + \frac{3L^{2}}{\alpha_{k}}\mathbb{E}[\|\Bar{x}_{k+1} - \Bar{x}_{k}\|^{2}].
\end{align}
\vspace{-1cm}By \eqref{lem:e_k_t:eq}, we consider
\begin{align*}
\mathbb{E}[\|e^{i}_{k,\mathcal{T}}\|^{2}]&\leq (1-\alpha_{k})^{\mathcal{T}}\mathbb{E}[\|e^{i}_{k,0}\|^{2}] + \alpha_{k}^2\sigma^{2}\sum_{t=0}^{\mathcal{T}-1}(1-\alpha_{k})^{\mathcal{T}-1-t}\notag\\
&\quad  + \alpha_{k}L^{2}\sum_{t = 0}^{\mathcal{T}-1}(1-\alpha_{k})^{\mathcal{T}-t-1}\mathbb{E}[\|x^{i}_{k,t}-\bar{x}_{k}\|^{2}]\notag\\
    &\leq (1-\alpha_{k})\mathbb{E}[\|e^{i}_{k,0}\|^{2}] + \sigma^{2}\mathcal{T}\alpha^{2}_{k} \notag\\
        &\quad  + \alpha_{k}L^{2}\sum_{t = 0}^{\mathcal{T}-1}\mathbb{E}[\|x^{i}_{k,t}-\bar{x}_{k}\|^{2}]\notag\\
        &\leq (1-\alpha_{k})\mathbb{E}[\|e^{i}_{k,0}\|^{2}]  + L^{4}T^3\alpha_{k}\beta_{k}^{2}\mathbb{E}[\|\Bar{x}_{k}-x^{\star}_{\mathcal{H}}\|^{2}\notag\\
        &\quad   + 2L^2\mathcal{T}\alpha_{k}\beta_{k}^{2}\sum_{t=0}^{\mathcal{T}}\|e^{i}_{k,t}\|^{2}]+ \sigma^{2}\mathcal{T}\alpha^{2}_{k},
\end{align*}
where the last inequality is due to \eqref{lem:x_k_t_minus_x_bar_squared:ineq}. Using this equation and $\big(1+\frac{\alpha_{k}}{2}\big) \leq 3/2$ we obtain
\begin{align*}
    &\big(1+\frac{\alpha_{k}}{2}\big)\frac{1}{|\mathcal{H}|}\sum_{i \in \mathcal{H}}\mathbb{E}[\|e^{i}_{k,\mathcal{T}}\|^{2}]\notag\\
    &\leq\big(1+\frac{\alpha_{k}}{2}\big)\frac{1}{|\mathcal{H}|}\sum_{i \in \mathcal{H}}(1-\alpha_{k})\mathbb{E}[\|e^{i}_{k,0}\|^{2}] + 2\sigma^{2}\mathcal{T}\alpha^{2}_{k}\notag\\ 
    &\quad + 2  L^{4}T^3\alpha_{k}\beta_{k}^{2}\mathbb{E}[\|\Bar{x}_{k}-x^{\star}_{\mathcal{H}}\|^{2}]\notag\\
    &\quad   + 3L^2\mathcal{T}\alpha_{k}\beta_{k}^{2}\frac{1}{|\mathcal{H}|}\sum_{i \in \mathcal{H}}\sum_{t=0}^{\mathcal{T}}\mathbb{E}[\|e^{i}_{k,t}\|^{2}]\allowdisplaybreaks\notag\\
    &\leq\big(1-\frac{\alpha_{k}}{2}\big)\mathbb{E}[W_{k}] + 2\sigma^{2}\mathcal{T}\alpha^{2}_{k}\notag\\ 
    &\quad + 2  L^{4}T^3\alpha_{k}\beta_{k}^{2}\mathbb{E}[\|\Bar{x}_{k}-x^{\star}_{\mathcal{H}}\|^{2}]\notag\\
    &\quad   + 3L^2\mathcal{T}\alpha_{k}\beta_{k}^{2}\frac{1}{|\mathcal{H}|}\sum_{i \in \mathcal{H}}\sum_{t=0}^{\mathcal{T}}\mathbb{E}[\|e^{i}_{k,t}\|^{2}],
\end{align*}
where the last inequality we use 
\begin{align*}
   (1-\alpha_{k})\big(1+\frac{\alpha_{k}}{2}\big) \leq 1-\frac{\alpha_{k}}{2}\cdot
\end{align*}
Thus, substitute the relation above into \eqref{eq:V_2_r0} we have
\begin{align} 
    \mathbb{E}[W_{k+1}] &\leq \big(1-\frac{\alpha_{k}}{2}\big)\mathbb{E}[W_{k}] + 2\sigma^{2}\mathcal{T}\alpha^{2}_{k}\notag\\ 
    &\quad + 2  L^{4}\mathcal{T}^3\alpha_{k}\beta_{k}^{2}\mathbb{E}[\|\Bar{x}_{k}-x^{*}_{\mathcal{H}}\|^{2}]\notag\\
    &\quad   + 3L^2\mathcal{T}\alpha_{k}\beta_{k}^{2}\frac{1}{|\mathcal{H}|}\sum_{i \in \mathcal{H}}\sum_{t=0}^{\mathcal{T}}\mathbb{E}[\|e^{i}_{k,t}\|^{2}]\notag\\ 
    &\quad + \frac{3L^{2}}{\alpha_{k}}\mathbb{E}[\|\Bar{x}_{k+1} - \Bar{x}_{k}\|^{2}]\notag\\
    &\leq \big(1-\frac{\alpha_{k}}{2}\big)\mathbb{E}[W_{k}] + 2\sigma^{2}\mathcal{T}\alpha^{2}_{k} \notag\\
    &\quad + 2  L^{4}\mathcal{T}^3\alpha_{k}\beta_{k}^{2}\mathbb{E}[\|\Bar{x}_{k}-x^{*}_{\mathcal{H}}\|^{2}]\notag\\
    &\quad   + 6L^2\mathcal{T}^2\alpha_{k}\beta_{k}^{2}\mathbb{E}[W_{k}] + 6L^2\mathcal{T}^2\sigma^{2}\alpha_{k}^2\beta_{k}^{2}\notag\\ 
    &\quad   + 12L^6\mathcal{T}^4\alpha_{k}\beta_{k}^{4}\mathbb{E}[\|\bar{x}_{k}-x_{\mathcal{H}}^{\star}\|^2]\notag\\ 
    &\quad + \frac{3L^{2}}{\alpha_{k}}\mathbb{E}[\|\Bar{x}_{k+1} - \Bar{x}_{k}\|^{2}]\notag\\
    &= \big(1-\frac{\alpha_{k}}{2}+ 6L^2\mathcal{T}^2\alpha_{k}\beta_{k}^{2}\big)\mathbb{E}[W_{k}]\notag\\ 
    &\quad    + 2\sigma^{2}\mathcal{T}\alpha^{2}_{k} + 6L^2\mathcal{T}^2\sigma^{2}\alpha_{k}^2\beta_{k}^{2}\notag\\ 
    &\quad  + \big(2  L^{4}\mathcal{T}^3\alpha_{k}\beta_{k}^{2} + 12L^6\mathcal{T}^4\alpha_{k}\beta_{k}^{4}\big)\mathbb{E}[\|\bar{x}_{k}-x_{\mathcal{H}}^{\star}\|^2]\notag\\ 
    &\quad + \frac{3L^{2}}{\alpha_{k}}\mathbb{E}[\|\Bar{x}_{k+1} - \Bar{x}_{k}\|^{2}],\notag
\end{align}
where we use \eqref{lem:e_k_t_V12:ineq} to obtain the second inequality.  
Next, applying Lemma \ref{lem:xbar_k_k+1}, the above inequality becomes
\begin{align}
    \mathbb{E}[W_{k+1}] &\leq \big(1-\frac{\alpha_{k}}{2}\big)\mathbb{E}[W_{k}]\notag\\
    &\quad +\big(6L^{2}\mathcal{T}^{2}\alpha_{k}\beta^{2}_{k} + \frac{120L^{2}\mathcal{T}^{2}\beta^{2}_{k}}{\alpha_{k}}\big)\mathbb{E}[W_{k}]\notag\\
    &\quad +\big(2L^{4}\mathcal{T}^{3}\alpha_{k}\beta^{2}_{k} + 12L^{6}\mathcal{T}^{4}\alpha_{k}\beta^{4}_{k} \big)\mathbb{E}[\|\bar{x}_{k}-x^{\star}_{\mathcal{H}}\|^{2}]\notag\\
    &\quad + \frac{300L^{4}\mathcal{T}^{2}\beta^{2}_{k}}{\alpha_{k}}\mathbb{E}[\|\bar{x}_{k}-x^{\star}_{\mathcal{H}}\|^{2}]+ 2\sigma^{2}\mathcal{T}\alpha^{2}_{k}\notag\\
    &\quad + 96L^{2}\mathcal{T}^{2}\sigma^{2}\beta^{2}_{k}+ 6L^{2}\mathcal{T}^{2}\sigma^{2}\alpha^{2}_{k}\beta^{2}_{k} \notag\\
    &\quad + \frac{96L^{2}\mathcal{T}^{2}\sigma^{2}f\alpha_{k}\beta^{2}_{k}}{|\mathcal{H}|}.\notag
\end{align}
Using $\beta_k \leq \alpha_{k}$ and $\beta_{k}L\mathcal{T} \leq 1$  we obtain
\begin{align}
     \mathbb{E}[W_{k+1}] &\leq \big(1-\frac{\alpha_{k}}{2}\big)\mathbb{E}[W_{k}]\notag\\
    &\quad + \big(6L\mathcal{T}\beta_{k} + 120L^{2}\mathcal{T}^{2}\beta_{k}\big)\mathbb{E}[W_{k}]\notag\\
    &\quad +\big(14L^{3}\mathcal{T}^{2}\alpha_{k}\beta_{k} + \frac{300L^{4}\mathcal{T}^{2}\beta^{2}_{k}}{\alpha_{k}} \big)\mathbb{E}[\|\bar{x}_{k}-x^{\star}_{\mathcal{H}}\|^{2}]\notag\\
    &\quad+ 2\sigma^{2}\mathcal{T}\alpha^{2}_{k} + 96L^{2}\mathcal{T}^{2}\sigma^{2}\beta^{2}_{k}\notag\\
    &\quad+ 6L^{2}\mathcal{T}^{2}\sigma^{2}\alpha^{2}_{k}\beta^{2}_{k} + \frac{96L^{2}\mathcal{T}^{2}\sigma^{2}f\alpha_{k}\beta^{2}_{k}}{|\mathcal{H}|}\allowdisplaybreaks\notag\\ 
    &\leq \big(1-\frac{\alpha_{k}}{2} + 126(L+1)^{2}\mathcal{T}^{2}\beta_{k}\big)\mathbb{E}[W_{k}]\notag\\
    &\quad +\big(14L^{3}\mathcal{T}^{2}\alpha_{k}\beta_{k} + \frac{300L^{4}\mathcal{T}^{2}\beta^{2}_{k}}{\alpha_{k}} \big)\mathbb{E}[\|\bar{x}_{k}-x^{\star}_{\mathcal{H}}\|^{2}]\notag\\
     &\quad+ 2\sigma^{2}\mathcal{T}\alpha^{2}_{k} + 96L^{2}\mathcal{T}^{2}\sigma^{2}\beta^{2}_{k}\notag\\
    &\quad+ 6L^{2}\mathcal{T}^{2}\sigma^{2}\alpha^{2}_{k}\beta^{2}_{k} + \frac{96L^{2}\mathcal{T}^{2}\sigma^{2}f\alpha_{k}\beta^{2}_{k}}{|\mathcal{H}|}\notag\cdot
\end{align}

\subsection*{Proof of Lemma \ref{lem:Lemma_V1_SC}}
    Using \eqref{eq:x_bar_update} we have
    \begin{align}\label{eq:x_bar_squared_update}
        &\|\Bar{x}_{k+1} -x^{*}_{\mathcal{H}}\|^{2}\notag \\
        &= \|\bar{x}_{k}-x^{*}_{\mathcal{H}} - \mathcal{T}\beta_{k}\nabla q_{\mathcal{H}}(\bar{x}_{k})\|^{2} + \Big\|\mathcal{E}_{x} -\frac{\beta_{k}}{|\mathcal{H}|}\sum_{i \in \mathcal{H}}\sum_{l=0}^{\mathcal{T}-1}e^{i}_{k,l}\Big\|^{2} \notag\\
        &\quad -2\big(\bar{x}_{k}-x^{*}_{\mathcal{H}} - \mathcal{T}\beta_{k}\nabla q_{\mathcal{H}}(\bar{x}_{k})\big)^{T}\big(\mathcal{E}_{x} -\frac{\beta_{k}}{|\mathcal{H}|}\sum_{i \in \mathcal{H}}\sum_{l=0}^{\mathcal{T}-1}e^{i}_{k,l}\big)\notag\\
      &= P_{1} + P_{2} + P_{3} ,
    \end{align}
where $P_i$, for $i = 1, 2, 3$, are defined in that order. Firstly, using $\nabla q_{\mathcal{H}}(x^{\star}_{\mathcal{H}}) = 0$ along with Assumption \ref{Assum:Assumption_strong_convexity} and \ref{eq:Assumption_Lipschitz} we analyze the term $P_{1}$ as

    \begin{align}\label{eq:P1}
        &\|\bar{x}_{k}-x^{\star}_{\mathcal{H}} - \mathcal{T}\beta_{k}\nabla q_{\mathcal{H}}(\bar{x}_{k})\|^{2}\notag\\
        &= \|\bar{x}_{k}-x^{\star}_{\mathcal{H}}\|^{2}-2\mathcal{T}\beta_{k}\nabla q_{\mathcal{H}}(\bar{x}_{k})^{T}(\bar{x}_{k}-x^{\star}_{\mathcal{H}}) \notag\\
        &\quad+ \mathcal{T}^{2}\beta^{2}_{K}\|\nabla q_{\mathcal{H}}(\bar{x}_{k})\|^{2}\notag\\
         &= \|\bar{x}_{k}-x^{\star}_{\mathcal{H}}\|^{2}-2\mathcal{T}\beta_{k}(\nabla q_{\mathcal{H}}(\bar{x}_{k})-\nabla q_{\mathcal{H}}(x^{\star}_{\mathcal{H}}))^{T}(\bar{x}_{k}-x^{\star}_{\mathcal{H}}) \notag\\
        &\quad+ \mathcal{T}^{2}\beta^{2}_{K}\|\nabla q_{\mathcal{H}}(\bar{x}_{k})-\nabla q_{\mathcal{H}}(x^{\star}_{\mathcal{H}})\|^{2}\notag\\
        &\leq (1-2\mu\mathcal{T}\beta_{k} + L^{2}\mathcal{T}^{2}\beta^{2}_{k})\|\bar{x}_{k}-x^{\star}_{\mathcal{H}}\|^{2}
    \end{align}

    Secondly, using Cauchy-Schwarz inequality term $P_2$ can be expressed as
    \begin{align}
        &\Big\|\mathcal{E}_{x} -\frac{\beta_{k}}{|\mathcal{H}|}\sum_{i \in \mathcal{H}}\sum_{l=0}^{\mathcal{T}-1}e^{i}_{k,l}\Big\|^{2}\notag\\
        &\leq 2\|\mathcal{E}_{x}\|^{2}+ 2\beta^{2}_{k}\mathcal{T}\big(\frac{1}{|\mathcal{H}|}\sum_{i \in \mathcal{H}}\sum_{l=0}^{\mathcal{T}-1}\|e^{i}_{k,l}\|^{2}\big).\notag
    \end{align}

Taking the expectation on both sides of the above inequality, and applying Lemma \ref{lem:e_k_t_V12} along with \eqref{eq:E_x_last} from Lemma \ref{lem:xbar_k_k+1} with $\alpha_{k} \leq 1$ , we obtain the following

\begin{align}\label{eq:P2}
    &\mathbb{E}\Bigg[\Big\|\mathcal{E}_{x} -\frac{\beta_{k}}{|\mathcal{H}|}\sum_{i \in \mathcal{H}}\sum_{l=0}^{\mathcal{T}-1}e^{i}_{k,l}\Big\|^{2}\Bigg]\notag\\
   & \leq \big(\frac{16L^{2}\mathcal{T}^{2}\beta^{2}_{k}|\mathcal{B}_{k}|}{|\mathcal{H}|} + \frac{64L^{3}\mathcal{T}^{3}\beta^{3}_{k}|\mathcal{B}_{k}|}{|\mathcal{H}|} \big)\mathbb{E}[\|\bar{x}_{k} -x^{\star}_{\mathcal{H}}\|^{2}]\notag\\
   &\quad + 8L^{4}\mathcal{T}^{4}\beta^{4}_{k}\mathbb{E}[\|\bar{x}_{k} -x^{\star}_{\mathcal{H}}\|^{2}] \notag\\
   &\quad +\big(\frac{32\mathcal{T}^{2}\beta^{2}_{k}|\mathcal{B}_{k}|}{|\mathcal{H}|} + 4\mathcal{T}^{2}\beta^{2}_{k}\big)\mathbb{E}[W_{k}]\notag\\
   &\quad + 4\mathcal{T}^{2}\sigma^{2}\alpha_{k}\beta^{2}_{k} + \frac{32\mathcal{T}^{2}\sigma^{2}f\alpha^{2}_{k}\beta^{2}_{k}}{|\mathcal{H}|}\notag\\
   &\leq 88L^{2}\mathcal{T}^{2}\beta^{2}_{k}\mathbb{E}[\|\bar{x}_{k} -x^{\star}_{\mathcal{H}}\|^{2}] + \frac{36\mathcal{T}\beta_{k}}{\mu}\mathbb{E}[W_{k}]\notag\\
   &\quad + 4\mathcal{T}^{2}\sigma^{2}\alpha_{k}\beta^{2}_{k} + \frac{32\mathcal{T}^{2}\sigma^{2}f\alpha^{2}_{k}\beta^{2}_{k}}{|\mathcal{H}|},
\end{align}
where the last inequality is obtained using $\beta_{k}L\mathcal{T} \leq 1$ and $\mu \leq L$.
Thirdly we express term $P_3$ from \eqref{eq:x_bar_squared_update} as  
\begin{align}\label{eq:P3_v0}
    &-2\big(\bar{x}_{k}-x^{\star}_{\mathcal{H}} - \mathcal{T}\beta_{k}\nabla q_{\mathcal{H}}(\bar{x}_{k})\big)^{T}\big(\mathcal{E}_{x} -\frac{\beta_{k}}{|\mathcal{H}|}\sum_{i \in \mathcal{H}}\sum_{l=0}^{\mathcal{T}-1}e^{i}_{k,l}\big)\notag\\
    &= -2\big(\bar{x}_{k}-x^{\star}_{\mathcal{H}} - \mathcal{T}\beta_{k}\nabla q_{\mathcal{H}}(\bar{x}_{k})\big)^{T}\mathcal{E}_{x}\notag\\
    &\quad +2\big(\bar{x}_{k}-x^{\star}_{\mathcal{H}} - \mathcal{T}\beta_{k}\nabla q_{\mathcal{H}}(\bar{x}_{k})\big)^{T}\big(\frac{\beta_{k}}{|\mathcal{H}|}\sum_{i \in \mathcal{H}}\sum_{l=0}^{\mathcal{T}-1}e^{i}_{k,l}\big)\notag\\
    &= P_{3a} + P_{3b},
\end{align}
where $P_{3a}$ and $ P_{3b}$ are defined in that order. 
Next, we apply the Cauchy-Schwarz inequality $2a^{T}b \leq \eta\|a\|^{2} + \frac{\|b\|^{2}}{\eta}$ for any $\eta >0$, and use Assumptions \ref{Assum:Assumption_strong_convexity} and \ref{eq:Assumption_Lipschitz} to analyze term $P_{3a}$. Thus, we have
\begin{align}
    &-2\big(\bar{x}_{k}-x^{\star}_{\mathcal{H}} - \mathcal{T}\beta_{k}\nabla q_{\mathcal{H}}(\bar{x}_{k})\big)^{T}\mathcal{E}_{x}\notag\\
    & \leq \frac{3L\mathcal{T}\beta_{k}|\mathcal{B}_{k}|}{|\mathcal{H}|} \|\bar{x}_{k}-x^{\star}_{\mathcal{H}} - \mathcal{T}\beta_{k}\nabla q_{\mathcal{H}}(\bar{x}_{k})\|^{2} + \frac{|\mathcal{H}|}{3L\mathcal{T}\beta_{k}|\mathcal{B}_{k}|}\|\mathcal{E}_{x}\|^{2}\notag\\
    \notag\\
    & \leq \frac{3L\mathcal{T}\beta_{k}|\mathcal{B}_{k}|}{|\mathcal{H}|} (\|\bar{x}_{k}-x^{\star}_{\mathcal{H}}\|^{2}+L^{2}\mathcal{T}^{2}\beta_{k}^{2}\|\bar{x}_{k}-x^{\star}_{\mathcal{H}}\|^{2}) \notag\\
    &\quad+ \frac{|\mathcal{H}|}{3L\mathcal{T}\beta_{k}|\mathcal{B}_{k}|}\|\mathcal{E}_{x}\|^{2}\notag.
\end{align}
 Next taking expectation on both sides of the above inequality and using \eqref{eq:E_x_last} from Lemma \ref{lem:xbar_k_k+1} along with $\alpha_{k} \leq 1$, we have
\begin{align}\label{eq:P_3a}
     &-2\mathbb{E}\Big[\big(\bar{x}_{k}-x^{\star}_{\mathcal{H}} - \mathcal{T}\beta_{k}\nabla q_{\mathcal{H}}(\bar{x}_{k})\big)^{T}\mathcal{E}_{x}\Big]\notag\\
     &\leq\big(\frac{17L\mathcal{T}\beta_{k}|\mathcal{B}_{k}|}{3|\mathcal{H}|}+\frac{32L^{2}\mathcal{T}^{2}\beta_{k}^{2}}{3}+\frac{3L^{3}\mathcal{T}^{3}\beta_{k}^{3}|\mathcal{B}_{k}|}{|\mathcal{H}|}\big)\mathbb{E}[\|\bar{x}_{k}-x^{\star}_{\mathcal{H}}\|^{2}] \notag\\
    &\quad + \frac{16\beta_{k}\mathcal{T}}{3L}\mathbb{E}[W_{k}] + \frac{16\mathcal{T}\sigma^{2}\alpha^{2}_{k}\beta_{k}}{3L}\notag\\
    &\leq \big(\frac{17L\mathcal{T}\beta_{k}|\mathcal{B}_{k}|}{3|\mathcal{H}|}+ 14L^{2}\mathcal{T}^{2}\beta^{2}_{k}\big)\mathbb{E}[\|\bar{x}_{k}-x^{\star}_{\mathcal{H}}\|^{2}] \notag\\
    &\quad + \frac{16\mathcal{T}\beta_{k}}{3\mu}\mathbb{E}[W_{k}] + \frac{16\mathcal{T}\sigma^{2}\alpha^{2}_{k}\beta_{k}}{3\mu},
\end{align}
where the last inequality is obtained using $\beta_{k}L\mathcal{T} \leq 1$, $\frac{|\mathcal{B}_{k}|}{|\mathcal{H}|} \leq 1$ and $\mu \leq L$. To analyze the term $P_{3b}$ from \eqref{eq:P3_v0}, we use the Cauchy-Schwarz along with Assumption \ref{Assum:Assumption_strong_convexity} and \ref{eq:Assumption_Lipschitz} to obtain
\begin{align}
    &2\big(\bar{x}_{k}-x^{\star}_{\mathcal{H}} - \mathcal{T}\beta_{k}\nabla q_{\mathcal{H}}(\bar{x}_{k})\big)^{T}\big(\frac{\beta_{k}}{|\mathcal{H}|}\sum_{i \in \mathcal{H}}\sum_{l=0}^{\mathcal{T}-1}e^{i}_{k,l}\big)\notag\\
    &\leq \frac{\mu\mathcal{T}\beta_{k}}{18}\|\bar{x}_{k}-x^{\star}_{\mathcal{H}} -\mathcal{T}\beta_{k}\nabla q_{\mathcal{H}}(\bar{x}_{k})\|^{2}\notag\\
    &\quad + \frac{18\beta_{k}}{\mu}\big(\frac{1}{|\mathcal{H}|}\sum_{i \in \mathcal{H}}\sum_{l=0}^{\mathcal{T}-1}\|e^{i}_{k,l}\|^{2}\big)\notag\\
    &\leq \frac{\mu\mathcal{T}\beta_{k}}{18}(\|\bar{x}_{k}-x^{\star}_{\mathcal{H}}\|^{2} + L^{2}\mathcal{T}^{2}\beta^{2}_{k}\|\bar{x}_{k}-x^{\star}_{\mathcal{H}}\|^{2})\notag\\
    &\quad+ \frac{18\beta_{k}}{\mu}\big(\frac{1}{|\mathcal{H}|}\sum_{i \in \mathcal{H}}\sum_{l=0}^{\mathcal{T}-1}\|e^{i}_{k,l}\|^{2}\big)\notag.
\end{align}
Next, taking expectation on both the sides of the above inequality along with the result from Lemma \ref{lem:e_k_t_V12} we have 
\begin{align}\label{eq:P_3b}
    &2\mathbb{E}\Big[\big(\bar{x}_{k}-x^{\star}_{\mathcal{H}} - \mathcal{T}\beta_{k}\nabla q_{\mathcal{H}}(\bar{x}_{k})\big)^{T}\big(\frac{\beta_{k}}{|\mathcal{H}|}\sum_{i \in \mathcal{H}}\sum_{l=0}^{\mathcal{T}-1}e^{i}_{k,l}\big)\Big]\notag\\
    &\leq \big(\frac{\mu\mathcal{T}\beta_{k}}{18} + \frac{\mu L^{2}\mathcal{T}^{3}\beta^{3}_{k}}{18}+\frac{72L^{4}\mathcal{T}^{3}\alpha_{k}\beta^{3}_{k}}{\mu}\big)\mathbb{E}[\|\bar{x}_{k}-x^{\star}_{\mathcal{H}}\|^{2}]\notag\\
    &\quad +\frac{36\beta_{k}\mathcal{T}}{\mu}\mathbb{E}[W_{k}] + \frac{36\mathcal{T}\sigma^{2}\alpha^{2}_{k}\beta_{k}}{\mu}\notag\\
    &\leq \big(\frac{\mu\mathcal{T}\beta_{k}}{18} + \frac{ L^{2}\mathcal{T}^{2}\beta^{2}_{k}}{18}+\frac{72L^{2}\mathcal{T}\alpha_{k}\beta_{k}}{\mu}\big)\mathbb{E}[\|\bar{x}_{k}-x^{\star}_{\mathcal{H}}\|^{2}]\notag\\
    &\quad +\frac{36\beta_{k}\mathcal{T}}{\mu}\mathbb{E}[W_{k}] + \frac{36\mathcal{T}\sigma^{2}\alpha^{2}_{k}\beta_{k}}{\mu},
\end{align}
where the last inequality is obtained using $\mu\leq L$ and $\beta_{k}L\mathcal{T} \leq 1$. Putting back the results from \eqref{eq:P_3a} and \eqref{eq:P_3b} back into \eqref{eq:P3_v0}, we have
\begin{align}\label{eq:P3_final}
     &-2\mathbb{E}\Big[\big(\bar{x}_{k}-x^{\star}_{\mathcal{H}} - \mathcal{T}\beta_{k}\nabla q_{\mathcal{H}}(\bar{x}_{k})\big)^{T}\big(\mathcal{E}_{x} -\frac{\beta_{k}}{|\mathcal{H}|}\sum_{i \in \mathcal{H}}\sum_{l=0}^{\mathcal{T}-1}e^{i}_{k,l}\big)\Big]\notag\\
     &\leq \big(\frac{\mu\mathcal{T}\beta_{k}}{18}+\frac{17L\mathcal{T}\beta_{k}|\mathcal{B}_{k}|}{3|\mathcal{H}|}  \big)\mathbb{E}[\|\bar{x}_{k}-x^{\star}_{\mathcal{H}}\|^{2}]\notag\\
     &\quad + \big(15L^{2}\mathcal{T}^{2}\beta^{2}_{k} + \frac{72L^{2}\mathcal{T}\alpha_{k}\beta_{k}}{\mu}\big)\mathbb{E}[\|\bar{x}_{k}-x^{\star}_{\mathcal{H}}\|^{2}]\notag\\
     &\quad +\frac{42\mathcal{T}\beta_{k}}{\mu}\mathbb{E}[W_{k}] + \frac{42\mathcal{T}\sigma^{2}\alpha^{2}_{k}\beta_{k}}{\mu}.
\end{align}
Finally, by taking the expectation on both sides of \eqref{eq:x_bar_squared_update}, substituting the expressions for \( P_1 \), \( P_2 \), and \( P_3 \) from \eqref{eq:P1}, \eqref{eq:P2}, and \eqref{eq:P3_final}, respectively we obtain

\begin{align*}
    &\mathbb{E}[\|\bar{x}_{k+1}-x^{\star}_{\mathcal{H}}\|^{2}] \notag\\
    & \leq \big(1-\frac{35\mu\mathcal{T}\beta_{k}}{18} + \frac{17L\mathcal{T}\beta_{k}|\mathcal{B}_{k}|}{3|\mathcal{H}|}+ 103L^{2}\mathcal{T}^{2}\beta^{2}_{k}\big)\mathbb{E}[\|\bar{x}_{k}-x^{\star}_{\mathcal{H}}\|^{2}]\notag\\
    &\quad + \frac{72L^{2}\mathcal{T}\alpha_{k}\beta_{k}}{\mu}\mathbb{E}[\|\bar{x}_{k}-x^{\star}_{\mathcal{H}}\|^{2}] +  \frac{78\mathcal{T}\beta_{k}}{\mu}\mathbb{E}[W_{k}]\notag\\
    &\quad +4\mathcal{T}^{2}\sigma^{2}\alpha_{k}\beta^{2}_{k} + \frac{48\mathcal{T}\sigma^{2}\alpha^{2}_{k}\beta_{k}}{\mu} + \frac{32\mathcal{T}^{2}\sigma^{2}f\alpha^{2}_{k}\beta^{2}_{k}}{|\mathcal{H}|}\notag.
\end{align*}
This concludes our proof.

\subsection*{Proof of Lemma \ref{lem:Lemma_V1_PL}}
To prove lemma \ref{lem:Lemma_V1_PL} we require the following lemma.
\begin{lemma}\label{lem:norm_E_x}
    \begin{align}
       \hspace{-0.3cm}\|\mathcal{E}_{k}\|\!\! \leq\!\! \frac{2L\mathcal{T}\beta_{k}|\mathcal{B}_{k}|}{\mu|\mathcal{H}|}\|\nabla q_{\mathcal{H}}(\bar{x}_{k})\| + 2\beta_{k}\big(\frac{1}{|\mathcal{H}|}\!\sum_{i \in \mathcal{H}}\!\!\sum_{t=0}^{\mathcal{T}-1}\|e^{i}_{k,t}\|\big).
    \end{align}
\end{lemma}
\begin{proof}
    Using the definition of $\mathcal{E}_{k}$ from \eqref{eq:V_x} we have
    \begin{align*}
        \|\mathcal{E}_{x}\| \leq \frac{1}{|\mathcal{H}|}\sum_{i \in \mathcal{B}_{k}}\|x^{i}_{k,\mathcal{T}}-\bar{x}_{k}\| + \frac{1}{|\mathcal{H}|}\sum_{i \in \mathcal{H}\backslash\mathcal{H}_{k}}\|x^{i}_{k,\mathcal{T}}-\bar{x}_{k}\|.
    \end{align*}
    By \eqref{alg:sort_distances}, there exists $j \in \mathcal{H}\backslash \mathcal{H}_{k}$ such that $\|x^{i}_{k,t}-\Bar{x}_{k}\| \leq \|x^{j}_{k,t}-\Bar{x}_{k}\|$ for all agent $i \in \mathcal{B}_{k}$ using which the above inequality becomes
    \begin{align}\label{eq:norm_E_x_analysis_0}
        \|\mathcal{E}_{x}\| \leq \frac{|\mathcal{B}_{k}|}{|\mathcal{H}|}\|x^{j}_{k,\mathcal{T}}-\bar{x}_{k}\| +\frac{1}{|\mathcal{H}|}\sum_{i \in \mathcal{H}\backslash\mathcal{H}_{k}}\|x^{j}_{k,\mathcal{T}}-\bar{x}_{k}\|.
    \end{align}
Here, we consider that there exists an agent $j \in \mathcal{H}\backslash\mathcal{H}_{k}$ such that the quantity $ \|x^{j}_{k,t}-\Bar{x}_{k}\|$ is maximum over all the agents in the set $\mathcal{H}\backslash\mathcal{H}_{k}$. Further since $|\mathcal{H}\backslash\mathcal{H}_{k}| = |\mathcal{B}_{k}|$, we have 
\begin{align*}
     \|x^{j}_{k,t}-\Bar{x}_{k}\|^{2} &\leq \frac{1}{|\mathcal{H}\backslash\mathcal{H}_{k}|}\sum_{i \in \mathcal{H}\backslash\mathcal{H}_{k}}\|x^{i}_{k,t}-\Bar{x}_{k}\|^{2}\notag\\
     &=\frac{1}{|\mathcal{B}_{k}|}\sum_{i \in \mathcal{H}\backslash\mathcal{H}_{k}}\|x^{i}_{k,t}-\Bar{x}_{k}\|^{2}.
\end{align*} 
Next using \eqref{eq:redundancy_condition} which implies $\nabla q^{i}(x^{*}_{\mathcal{H}}) =0$ which  with the result from \eqref{eq:x_update_T}, we express \eqref{eq:norm_E_x_analysis_0} as
\begin{align}
    \|\mathcal{E}_{k}\| &\leq \frac{2}{|\mathcal{H}|}\sum_{\mathcal{H}\backslash\mathcal{H}_{k}}\|x^{i}_{k,\mathcal{T}}-\bar{x}_{k}\|\notag\\
    &\leq \frac{2}{|\mathcal{H}|}\sum_{\mathcal{H}\backslash\mathcal{H}_{k}}\big(\Big\|-\mathcal{T}\beta_{k}\nabla q^{i}(\bar{x}_{k})-\beta_{k}\sum_{t=0}^{\mathcal{T}-1}e^{i}_{k,t}\Big\|\big)\notag\\
    &\leq \frac{2\mathcal{T}\beta_{k}}{|\mathcal{H}|}\sum_{i \in \mathcal{H}\backslash\mathcal{H}_{k}}\|\nabla q^{i}(\bar{x}_{k})-\nabla q^{i}(x^{\star}_{\mathcal{H}})\|\notag\\
    &\quad +\frac{2\beta_{k}}{|\mathcal{H}|}\sum_{i \in \mathcal{H}\backslash\mathcal{H}_{k}}\sum_{t=0}^{\mathcal{T}-1}\|e^{i}_{k,t}\|\notag\\
    & \leq \frac{2L\mathcal{T}\beta_{k}|\mathcal{B}_{k}|}{|\mathcal{H}|}\|\bar{x}_{k}-x^{\star}_{\mathcal{H}}\| + 2\beta_{k}\big(\frac{1}{|\mathcal{H}|}\sum_{i \in \mathcal{H}\backslash\mathcal{H}_{k}}\sum_{t=0}^{\mathcal{T}-1}\|e^{i}_{k,t}\|\big)\notag\\
    & \leq \frac{2L\mathcal{T}\beta_{k}|\mathcal{B}_{k}|}{\mu|\mathcal{H}|}\|\nabla q_{\mathcal{H}}(\bar{x}_{k})\|+ 2\beta_{k}\big(\frac{1}{|\mathcal{H}|}\sum_{i \in \mathcal{H}}\sum_{t=0}^{\mathcal{T}-1}\|e^{i}_{k,t}\|\big),\notag
\end{align}
where the second last inequality is due to Assumptiom \ref{eq:Assumption_Lipschitz} and last inequality is obtained using Assumption \ref{eq:Assumption_PL_condition}.
\end{proof}

\subsection*{Proof of Lemma \ref{lem:Lemma_V1_PL}}

Assumption \ref{eq:Assumption_Lipschitz} implies $q_{\mathcal{H}}$ has Lipschitz continuous gradient using which we have
\begin{align}\label{eq:diff_q_h}
    & q_{\mathcal{H}}(\Bar{x}_{k+1})-  q_{\mathcal{H}}(\Bar{x}_{k})\notag\\
    &\leq \nabla q_{\mathcal{H}}(\Bar{x}_{k})^{T}(\Bar{x}_{k+1}-\Bar{x}_{k}) + \frac{L}{2}\|\Bar{x}_{k+1}-\Bar{x}_{k}\|^{2}.
\end{align}
Using \eqref{eq:x_bar_update}, we analyze the first term in the right hand side of the above equation as the following
\begin{align}\label{eq:diff_q_h_term_one}
    &\nabla q_{\mathcal{H}}(\Bar{x}_{k})^{T}(\Bar{x}_{k+1}-\Bar{x}_{k})\notag\\
    &= -\mathcal{T}\beta_{k}\|\nabla q_{\mathcal{H}}(\Bar{x}_{k})\|^{2} -\beta_{k}\nabla q_{\mathcal{H}}(\Bar{x}_{k})^{T}\big(\frac{1}{|\mathcal{H}|}\sum_{i \in \mathcal{H}}\sum_{t=0}^{\mathcal{T}-1}e^{i}_{k,t}\big) \notag\\
    &\quad+ \mathcal{E}_{x}^{T}\nabla q_{\mathcal{H}}(\Bar{x}_{k})\notag\\
    & = A_1 + A_2 +A_3,
\end{align}
where $A_{i}$ with $i=1,2,3$ are defined in that order. To analyze \(A_2\), we apply Assumption \ref{eq:Assumption_PL_condition} and utilize the Cauchy-Schwarz inequality. Thus, we have
\begin{align}
    &-\beta_{k}\nabla q_{\mathcal{H}}(\Bar{x}_{k})^{T}\big(\frac{1}{|\mathcal{H}|}\sum_{i \in \mathcal{H}}\sum_{t=0}^{\mathcal{T}-1}e^{i}_{k,t}\big) \notag\\
    & \leq \frac{\beta_{k}\mathcal{T}}{12}\|\nabla q_{\mathcal{H}}(\bar{x}_{k})\|^{2} + \frac{3\beta_{k}}{\mathcal{T}}\Big\|\frac{1}{|\mathcal{H}|}\sum_{i \in \mathcal{H}}\sum_{l=0}^{\mathcal{T}-1}e^{i}_{k,l}\Big\|^{2}\notag\\
    &\leq \frac{\beta_{k}\mathcal{T}}{12}\|\nabla q_{\mathcal{H}}(\bar{x}_{k})\|^{2} + 3\beta_{k}\big(\frac{1}{|\mathcal{H}|}\sum_{i \in \mathcal{H}}\sum_{l=0}^{\mathcal{T}-1}\|e^{i}_{k,l}\|^{2}\big).\notag
\end{align}
Taking expectation on both the sides of the above inequality and using Lemma \ref{lem:e_k_t_V12} along with Assumption \ref{eq:Assumption_PL_condition} we have
\begin{align}\label{eq:TermA_2}
    &-\mathbb{E}\Big[\beta_{k}\nabla q_{\mathcal{H}}(\Bar{x}_{k})^{T}\big(\frac{1}{|\mathcal{H}|}\sum_{i \in \mathcal{H}}\sum_{t=0}^{\mathcal{T}-1}e^{i}_{k,t}\big)\Big] \notag\\
    & \leq  \big(\frac{\beta_{k}\mathcal{T}}{12} + \frac{12L^{4}\mathcal{T}^{3}\beta^{3}_{k}}{\mu^{2}}\big)\mathbb{E}[\|\nabla q_{\mathcal{H}}(\bar{x}_{k})\|^{2}] + 6\beta_{k}\mathcal{T}\mathbb{E}[W_{k}] \notag\\
    &\quad+ 6\sigma^{2}\mathcal{T}\alpha_{k}\beta_{k}.
\end{align}
Next, to analyze term $A_3$ from \eqref{eq:diff_q_h_term_one} we use Lemma \ref{lem:norm_E_x} and Cauchy-Schwarz ineuqlity to obtain
\begin{align}
    \mathcal{E}_{x}^{T}\nabla q_{\mathcal{H}}(\Bar{x}_{k})& \leq \|\mathcal{E}_{x}\|\|\nabla q_{\mathcal{H}}(\Bar{x}_{k})\|\notag\\
    & \leq \frac{2L\mathcal{T}\beta_{k}|\mathcal{B}_{k}|}{\mu|\mathcal{H}|}\|\nabla q_{\mathcal{H}}(\Bar{x}_{k})\|^{2} \notag\\
    &\quad+ 2\beta_{k}\|\nabla q^{\mathcal{H}}(\bar{x}_{k})\|\big(\frac{1}{|\mathcal{H}|}\sum_{i \in \mathcal{H}}\sum_{l=0}^{\mathcal{T}-1}\|e^{i}_{k,l}\|\big)\notag\\
    &\leq \big(\frac{2L\mathcal{T}\beta_{k}|\mathcal{B}_{k}|}{\mu|\mathcal{H}|} + \frac{\mathcal{T}\beta_{k}}{12}\big)\|\nabla q_{\mathcal{H}}(\bar{x}_{k})\|^{2}\notag\\
    &\quad + 12\beta_{k}\big(\frac{1}{|\mathcal{H}|}\sum_{i \in \mathcal{H}}\sum_{t=0}^{\mathcal{T}-1}\|e^{i}_{k,t}\|^{2}\big).\notag
\end{align}
Taking expectation on both sides of the above inequality and using Lemma \ref{lem:e_k_t_V12} we have
\begin{align}\label{eq:TermA_3}
    &\mathbb{E}[\mathcal{E}_{x}^{T}\nabla q_{\mathcal{H}}(\bar{x}_{k})]\notag\\
    &\leq \Bigg(\frac{2L\mathcal{T}\beta_{k}|\mathcal{B}_{k}|}{\mu|\mathcal{H}|} + \frac{\mathcal{T}\beta_{k}}{12} + \frac{48L^{4}\mathcal{T}^{3}\beta^{3}_{k}}{\mu^{2}}\Bigg)\mathbb{E}[\|\nabla q_{\mathcal{H}}(\bar{x}_{k})\|^{2}] \notag\\
    &\quad + 24\mathcal{T}\beta_{k}\mathbb{E}[W_{k}] + 24\mathcal{T}\sigma^{2}\alpha_{k}\beta_{k}.
\end{align}
Finally putting back the relations from \eqref{eq:TermA_2} and \eqref{eq:TermA_3} back into \eqref{eq:diff_q_h_term_one} and using $L\mathcal{T}\beta_{k} \leq 1$ we get
\begin{align}\label{eq:diff_q_h_term_one_final}
    &\mathbb{E}[\nabla q_{\mathcal{H}}(\bar{x}_{k})^{T}(\bar{x}_{k+1}-\bar{x}_{k})]\notag\\
    &\leq\big(-\frac{5\beta_{k}\mathcal{T}}{6} +  \frac{2L\beta_{k}\mathcal{T}|\mathcal{B}_{k}|}{\mu|\mathcal{H}|} + \frac{60L^{3}\mathcal{T}^{2}\beta^{2}_{k}}{\mu^{2}}\big)\mathbb{E}[\|\nabla q_{\mathcal{H}}(\bar{x}_{k})\|^{2}] \notag\\
    &\quad + 30\beta_{k}\mathcal{T}\mathbb{E}[V^{k}_{2}]+30\sigma^{2}\mathcal{T}\alpha^{2}_{k}.
\end{align}
Next, using Lemma \ref{lem:x_k_t_minus_x_bar_squared}, we analyze the second term in the right hand side of \eqref{eq:diff_q_h}. Thus we have
\begin{align}\label{eq:diff_q_h_Term_two}
    &\frac{L}{2}\mathbb{E}[\|\bar{x}_{k+1}-\bar{x}_{k}\|^{2}]\notag\\
    &\leq 50L^{3}\mathcal{T}^{2}\beta^{2}_{k}\mathbb{E}[\|\bar{x}_{k}-x^{*}_{\mathcal{H}}\|^{2}]  + 20L\mathcal{T}^{2}\beta^{2}_{k}\mathbb{E}[W_{k}] \notag\\
    &\quad+ 16L\mathcal{T}^{2}\sigma^{2}\alpha_{k}\beta^{2}_{k} + \frac{16L\mathcal{T}^{2}\sigma^{2}f\alpha^{2}_{k}\beta^{2}_{k}}{|\mathcal{H}|}\notag\\
    &\leq \frac{50L^{3}\mathcal{T}^{2}\beta^{2}_{k}}{\mu^{2}}\mathbb{E}[\|\nabla q_{\mathcal{H}}(\bar{x}_{k})\|^{2}]  + 20\mathcal{T}\beta_{k}\mathbb{E}[W_{k}] \notag\\
    &\quad+ 16L\mathcal{T}^{2}\sigma^{2}\alpha_{k}\beta^{2}_{k} + \frac{16L\mathcal{T}^{2}\sigma^{2}f\alpha^{2}_{k}\beta^{2}_{k}}{|\mathcal{H}|},
\end{align}
where the last inequality is obtained using Assumption \ref{eq:Assumption_PL_condition} and the condition $L\mathcal{T}\beta_{k} \leq 1$. 

Putting back the results from \eqref{eq:diff_q_h_term_one_final} and \eqref{eq:diff_q_h_Term_two} back into \eqref{eq:diff_q_h} we immediately obtain \eqref{lem:Lemma_V1_PL:ineq}.
\end{document}